\documentclass{article}
\usepackage[utf8]{inputenc}
\usepackage[T1]{fontenc}
\usepackage[english]{babel}
\usepackage{enumitem}
\setlist{nosep} 

\usepackage{stackrel}

\let\OLDthebibliography\thebibliography
\renewcommand\thebibliography[1]{
  \OLDthebibliography{#1}
  \setlength{\parskip}{0pt}
  \setlength{\itemsep}{0pt plus 0.3ex}
}

\usepackage{comment}

\usepackage{mathtools}
\usepackage{amsfonts}
\usepackage{amssymb}
\usepackage{amsthm}
\usepackage[subnum]{cases}

\usepackage[colorlinks=true]{hyperref}
\hypersetup{urlcolor=blue, citecolor=red}

\numberwithin{equation}{section}

\newtheorem{theorem}{Theorem}[section]

\newtheorem{lemma}[theorem]{Lemma}
\newtheorem{proposition}[theorem]{Proposition}

\newtheorem{definition}[theorem]{Definition}
\theoremstyle{remark}
\newtheorem{remark}[theorem]{Remark}

\usepackage{fullpage}
\usepackage{graphicx}
\graphicspath{ {./Imagens_Final/} }
\usepackage{multicol} 
\usepackage{wrapfig}
\usepackage{pgf,tikz,pgfplots}
\usetikzlibrary{shapes,arrows} 
\usetikzlibrary{automata,positioning}
\usetikzlibrary{decorations.markings}
\tikzset{->-/.style={decoration={
			markings,
			mark=at position #1 with {\arrow{>}}},postaction={decorate}}}
\usepackage{float}

\newcommand{\N}{\mathbb{N}}

\newcommand{\R}{\mathbb{R}}
\newcommand{\C}{\mathbb{C}}

\newcommand{\G}{\mathcal{G}}

\newcommand{\norm}[3]{\|#1\|_{L^{#2}(#3)}}

\DeclareMathOperator{\sech}{sech}

\title{Stability transitions of NLS {action} ground-states on metric graphs}

\author{Francisco Agostinho, Sim\~ao Correia and
	Hugo Tavares}

\date{\today} 

\begin{document}
\maketitle

\begin{abstract}
We study the orbital stability of action ground-states of the nonlinear Schr\"odinger equation over two particular cases of metric graphs, the $\mathcal{T}$ and the tadpole graphs. We show the existence of \emph{stability transitions} near the $L^2$-critical exponent, a new dynamical feature of the nonlinear Schr\"odinger equation. More precisely, as the frequency $\lambda$ increases, the action ground-state transitions from stable to unstable and then back to stable (or vice-versa).

This result is complemented with the stability analysis of ground-states in the asymptotic cases of low/high frequency and weak/strong nonlinear interaction. Finally, we present a numerical simulation of the stability of action ground-states depending on the nonlinearity and the frequency parameter, which validates the aforementioned theoretical results.

\end{abstract}

\section{Introduction}

Let $\mathcal{G}$ be a (finite) metric graph\footnote{Roughly speaking, a set of finitely many edges, identified either with half-lines or finite intervals,  joined through vertices, where the distance between two points is defined as the length of the shortest path joining them. For more details, see \cite{berkolaiko2013introduction}.}. We consider the nonlinear Schr\"odinger equation over\footnote{The equation is to be satisfied over each edge of $\mathcal{G}$, together with a continuity condition on the vertices.} $\mathcal{G}$,
\begin{equation}\label{NLSE_Chapter5}\tag{NLS}
	i\partial_t v+ \partial^2_{xx} v+|v|^{p-2}v=0, \quad v=v(t,x),\quad v:\R\times \mathcal{G}\to \C,\quad p>2. \end{equation}
In physical applications, the metric graphs can represent either an optic network \cite{bolte2014many, burioni2001bose, dalfovo1999theory} (where $v$ is then the signal transmitted along that network) or a ramified trap for a Bose-Einstein condensate \cite{bulgakov2011symmetry, hung2011symmetric} (in which case, $v$ is the wave function of the condensate). 

Among all solutions, we focus on standing waves
\[
v(t,x)=e^{i\lambda t}u(x),
\]
where the real parameter $\lambda$ is usually called \textit{frequency} and $u$ defines the profile of the standing wave. It can be easily checked that $v$ is a standing wave if and only if the profile $u$ is a solution of
\begin{equation}\label{StationaryNLS_Chapter5}
	-u'' + \lambda u = |u|^{p-2}u\ \text{in}\ \G.
\end{equation}
This equation is to be interpreted in the weak-$H^1(\mathcal G)$-sense; its solutions corresponds to critical points of the action functional 
\begin{equation}\label{eq:actionfunctional}
	S_\lambda(u,\mathcal{G})=\frac{1}{2}\int_{\mathcal{G}}|{u'}|^2+\lambda |u|^2dx-\frac{1}{p}\int_{\mathcal{G}}|u|^pdx,\ \text{for}\ u\in H^1(\mathcal{G}),
\end{equation} 
 where $H^1(\mathcal{G})$ is the set of (complex-valued) functions $u$ which are $H^1$ on each edge and continuous at each vertex. A standard computation shows that weak solutions satisfy the equation in the classical sense on each edge, and satisfy Neumann-Kirchoff conditions at each vertex, see \cite{berkolaiko2013introduction} for more details.
 
 A nontrivial solution of \eqref{StationaryNLS_Chapter5} is called a \textit{bound-state}. As bound-states generate orbits of \eqref{NLSE_Chapter5}, it is natural to analyze their dynamical stability.

\begin{definition}[\textit{Orbital Stability}]
	Let $\G$ be a metric graph. The \textit{orbit} of a solution $u\in H^1(\mathcal{G})$ of \eqref{StationaryNLS_Chapter5} is the set\footnote{if $\mathcal{G}=\R$, we must include the translation invariance in the definition of orbit, $\mathcal{O}(u):=\{e^{i\lambda\theta}u(\cdot+y):\ \theta\in\R,\ y\in\R\}$, see \cite[Section 8.3]{cazenave2003semilinear}.} $\mathcal{O}(u):=\{e^{i\lambda\theta}u:\ \theta\in\R\}$. We say that a bound-state $u$ is \textit{orbitally stable} if, for every $\varepsilon>0$, there exists $\delta>0$ such that
	\[ d\left(\Phi(0),\mathcal{O}(u^\lambda)\right)<\delta\ \ \implies d\left(\Phi(t),\mathcal{O}(u^\lambda)\right)<\varepsilon,\ \forall t>0,
    \]
	where $\Phi(t)$ is the solution of \eqref{NLSE_Chapter5} with initial data $\Phi(0)$, and $d(\cdot,\cdot)$ is the $H^1$-distance.
	
    If $u$ is not orbitally stable, we say that it is (orbitally) unstable.
\end{definition}

The stability of bound-states  is generally a hard open problem. For this reason, one usually restricts to the class of \emph{action ground-state solutions} of \eqref{StationaryNLS_Chapter5}, i.e., bound-states that attain the 
 least action level
\begin{equation}\label{eq:leastactionlevel} \mathcal{S}_{\mathcal{G}}(\lambda)=\inf\left\{S_\lambda(u,\mathcal{G}),\ u\in H^1(\mathcal{G})\setminus\left\{0\right\}: u \text{ is a bound-state of } \eqref{StationaryNLS_Chapter5}\right\}.
\end{equation}
{The study of such solutions (whose existence and variational characterization might depend on the graph) has been the object of recent work, see \cite{AgostinhoCorreiaTavares,AgostinhoCorreiaTavares2,de2023notion}.}\footnote{{As a parallel comment, we would also like to point out the intense activity on the study of \emph{energy} ground-states, see for e.g. \cite{ Simonegrid2018,adami2015nls, adami2016threshold, 
adami2017negative, 
Borthwick2023,
cacciapuoti2018variational,  chang2023normalized,
deCoster_etal2024,SimoneTree2020,   dovetta2024nonuniquenessnormalizedgroundstates, dovetta2020uniqueness, MR3959930} and references therein. On graphs, in general, these two notions of ground-state are not equivalent, see \cite{AgostinhoCorreiaTavares, dovetta2024nonuniquenessnormalizedgroundstates, dovetta2023action}.}}
For metric graphs with relatively simple topologies, it can be shown that, given $\lambda>0$, there exists a unique action ground-state $u^\lambda$ up to phase multiplication.
In this work, we focus mainly on two such classes (see \cite{AgostinhoCorreiaTavares, AgostinhoCorreiaTavares2} and Theorem \ref{thm:uniqueness} below): \begin{itemize}
    \item the $\mathcal{T}-$graph, consisting of two half-lines, $h_1$ and $h_2$, and a terminal edge of length $\ell>0$, all attached at the same vertex, which we identify with $x=0$, see Figure \ref{fig:T}. To emphasize the dependence on $\ell>0$, throughout this paper we will denote this graph by $\mathcal{T}_\ell$;
    \item the tadpole graph, consisting of one half-line $h_1$ and one loop of length $2\ell>0$, see Figure \ref{fig:tadpole}. We denote such graph by $\mathcal{G}_\ell$.
\end{itemize}
\begin{figure}[h]
	\centering
	\begin{minipage}{0.45\textwidth}
		\centering\begin{tikzpicture}[node distance=2.5cm,  every loop/.style={}]
			\node(a)		{\Large $\infty$};
			\node[circle, fill=black, label=above:\large $\mathbf{0}$](b)[right of=a]{};
			\node[circle, fill=black, label=left: \large $\boldsymbol{\ell}$](c)[below of=b]{};
			\node(d)[right of=b]{\Large $\infty$};
			\path (a)edge node[above]{$h_1$}(b)
			(b)edge node[left]{$e_1$} (c)
			edge node[above]{$h_2$} (d)
			;	   	
		\end{tikzpicture}\label{fig:T}
		\caption{The class of $\mathcal{T}$-graphs.}\label{fig:T}	
	\end{minipage}
	\begin{minipage}{0.45\textwidth}
		\centering
		\begin{tikzpicture}[node distance=3cm,every loop/.style={}]
			\node(1){\Large $\infty$};
			\node[circle, fill=black, label=above:\large $\mathbf{0}$](3)[right of=1]{};
			\path
			(1) edge node[above]{$h_1$} (3)
			(3) edge[out=-30,in=30,loop,scale=6] node[left]{$e_1$} (3)
			;
		\end{tikzpicture}
		\vspace{7pt}
		\caption{The class of tadpole graphs.}\label{fig:tadpole}
	\end{minipage}
\end{figure}

In the euclidean cases $\mathcal{G}=\R$ or $\mathcal{G}=\R^+$, orbital stability of ground-states is completely characterized \cite{berestycki1981instability}, \cite{cazenave1982orbital}: regardless of the time-frequency parameter $\lambda>0$, for $p<6$, the ground-state is orbitally stable, and, for $p> 6$, it is unstable. Due to the scaling invariance of the spatial domain, the $L^2$-critical case $p=6$ marks a clear transition between both behaviors. These results have been extended to star-graphs, as scalings still preserve the domain \cite{Adami_star2012}. 
The stability of ground-states in other metric graphs has received some attention in the last ten years. See, for example, \cite{gustafson} on the circle; \cite{cacciapuoti2018variational} on compact graphs; \cite{kairzhan2021standing} on flower graphs; \cite{angulo2024stability} on looping edge graphs; \cite{noja2015bifurcations,  noja2020standing, pava2024stability}  on the tadpole graph; and references therein. Albeit being important steps towards a comprehensive theory, several of these works impose restrictions either in the power $p$ or in the frequency $\lambda$ and are unable to provide clear pictures of the orbital stability for action ground-states.

\medskip

For general metric graphs $\mathcal{G}$, the loss of scaling invariance raises the question of what is the critical threshold for stability/instability in terms of both $\lambda$ and $p$. In particular, it opens the possibility for \emph{stability transitions}: for fixed $p$, as the frequency increases, the nature of the action ground-state $u^\lambda$ varies from unstable to stable and back to unstable (or vice-versa).
\begin{definition}[Stability transition]
    Given a metric graph $\mathcal{G}$ and $p>2$, suppose that, for each $\lambda>0$, there exists a unique ground-state $u^\lambda$. We say that an USU-stability transition occurs if there exist frequencies $0<\lambda_1<\lambda_2<\dots<\lambda_6$ for which
		\begin{itemize}
			\item $u^\lambda$ is orbitally unstable for $\lambda\in[\lambda_1,\lambda_2]$.
			\item $u^\lambda$ is orbitally stable for $\lambda\in[\lambda_3,\lambda_4]$.
			\item $u^\lambda$ is orbitally unstable for $\lambda\in[\lambda_5,\lambda_6]$.
		\end{itemize}
    An analogous definition holds for SUS-stability transitions.
\end{definition}

Stability transitions represent  highly nonlinear phenomena and open some interesting applications, where the topology of the spatial domain can be used to influence the range of observable frequencies $\lambda$ in the related physical model. Up until now, the existence of stability transitions was completely unknown, even in the case of bounded euclidean domains (such as balls or rectangles).
Our first main result is the occurrence of stability transitions for the $\mathcal{T}$ and tadpole graphs near the $L^2$-critical case.

\begin{theorem}\label{thm:transition_particular}
Given $p>2$ and $\lambda,\ell>0$, let $u^\lambda$ be the unique positive action ground-state solution of \eqref{StationaryNLS_Chapter5} on either the $\mathcal{T}$ or tadpole graphs. Then, there exists $\varepsilon>0$ such that
	\begin{enumerate}
    \item  for $p\in(6,6+\varepsilon)$, there exists a USU-stability transition.
		\item for $p\in(6-\varepsilon,6)$, there exists a SUS-stability transition.
	
\item    For $p=6$, up to a single frequency $\lambda^*=\lambda^*(\mathcal{G})$,
        \begin{enumerate}
            \item for the $\mathcal{T}$-graph, there exists $\bar \lambda>0$ such that $u^\lambda$ is orbitally unstable for $\lambda\in (0,\bar\lambda)$ and orbitally stable for $\lambda\in (\bar\lambda,\infty)$.
            \item for the tadpole graph, there exists $\tilde \lambda>0$ such that $u^\lambda$ is orbitally stable for $\lambda\in (0,\tilde \lambda)$ and orbitally unstable for $\lambda\in (\tilde \lambda,\infty)$.
        \end{enumerate}
        \end{enumerate} 
\end{theorem}

\begin{remark}
We point out that part 3(b) of the above theorem follows from the conclusions of \cite{pelinovsky2021edge}. We decided to include a precise statement here for completeness.
\end{remark}

\medskip
\noindent
The proof of Theorem \ref{thm:transition} can be split into three steps:

\medskip
\begin{enumerate}
\item First, we check that the classical Grillakis-Shatah-Strauss \cite{grillakis1987stability,grillakis1990stability} conditions are verified in both classes of graphs. In particular, we show that, for each $\lambda>0$, $u^\lambda$ is a non-degenerate critical point of the action functional, with Morse index 1 . This allows to reduce of the study of the orbital stability of action ground-states to the monotonicity of the map 
\begin{equation}\label{eq:mass_function}
\lambda \mapsto\left\|u^\lambda\right\|_{L^2}
\end{equation}
through the \textit{Vakhitov--Kolokolov Stability Criterion}, see Theorem \ref{th:5.27} for more details.
\item Secondly, we provide a complete study of the map \eqref{eq:mass_function} in the $L^2$-critical case $p=6$, showing in particular it is not monotone, determining the monotonicity intervals and performing an asymptotic analysis as $\lambda\to 0^+$ and $\lambda\to +\infty$.
\item Finally, we show the continuity in the parameter $p>2$ of the quantity $\partial_\lambda \|u^\lambda\|_{L^2}^2$. 
\end{enumerate}

\smallbreak

This is actually a program on how to prove stability transitions close to $p\sim 6$. More precisely, under some basic assumptions on the ground-state curve $\lambda\mapsto u^\lambda$, we can prove the existence of stability transitions for general metric graphs (or even spatial domains), see Theorem \ref{thm:transition} below.

\begin{remark}
    The problematic frequency $\lambda^*$ in Theorem \ref{thm:transition_particular}-3 arises in the process of verifying the spectral conditions of the Grillakis-Shatah-Strauss theory (see Section \ref{AssumptionA_3} and, in particular, Proposition \ref{prop:5.26} for its explicit expression). We believe it is only a technical difficulty and should be overcome upon further analysis.
\end{remark}

\smallbreak
Theorem \ref{thm:transition_particular} shows stability transitions for $p\sim 6$. It is a natural question to ask what happens for $p\sim 2^+$ or $p\sim \infty$. This is the content of our second main result, which partially answers Conjecture 1.6 in \cite{AgostinhoCorreiaTavares}. The result also deals with the asymptotic cases $\lambda\sim 0$ and $\lambda\sim \infty$.

\begin{theorem}\label{thm:asymptotics}
    Consider either the $\mathcal{T}$-graph or the tadpole graph.
    Fix $\delta>0$.

\medskip
    \noindent\underline{Asymptotic stability in $\lambda$}
    
    \smallskip
    \begin{enumerate}
    \item if $p\in [2+\delta,6-\delta]$, action ground-states are orbitally stable for both $\lambda\sim 0$ and $\lambda\sim \infty$, uniformly in $p$.
    \item if $p\in[6+\delta, 1/\delta]$, action ground-states are orbitally unstable for both $\lambda\sim 0$ and $\lambda\sim \infty$, uniformly in $p$.
    \end{enumerate}

    \medskip
   \noindent\underline{Asymptotic stability in $p$}

     \smallskip
    \begin{enumerate}
    \item[(a)] for $p\sim 2^+$, action ground-states are orbitally stable, uniformly in $\lambda\in[\delta,1/\delta]$.
    \item[(b)] for $p\sim \infty$, action ground-states are orbitally unstable, uniformly in $\lambda\in[\delta,1/\delta]$.
    \end{enumerate}
\end{theorem}

The conclusions of Theorem \ref{thm:asymptotics} for fixed $p$ can be heuristically explained as follows. Observe that, for $\lambda=\ell^2>0$, the scaling $u(x)=\ell^\frac{2}{p-2}v(\ell x)=\lambda^\frac{1}{p-2}v(\sqrt{\lambda}x)$ satisfies
\[
-v''+v=|v|^{p-2}v \text{ in } \mathcal{T}_\ell \iff -u''+\lambda u=|u|^{p-2}u \text{ in } \mathcal{T}_1,
\]
and the same goes for the tadpole graph $\mathcal{G}_\ell$. Therefore, we can fix one of the parameters and vary only the other one; in particular, the statements in our two main results may be stated in terms of $\ell$ (and some of the proofs will actually be performed in such setting). The asymptotic behavior as $\lambda\to \infty$ and $\lambda\sim 0$ in Theorem \ref{thm:asymptotics} becomes quite natural when read in terms of $\ell$: for $\ell\sim 0,\ \infty$, the situation is close to having either a half-line or the whole line, and thus it is natural that the results are similar to those of these classical cases. 

\begin{remark}
    Theorem \ref{thm:transition_particular} deals with $p\sim 6$, while Theorem \ref{thm:asymptotics} analyzes the limiting cases in the $(p,\lambda)$ plane. The reason why we are unable to present the full stability characterization in $(p,\lambda)$ is merely a technical one. As we shall see in Section \ref{sec:explicit_mass}, we are able to derive a closed formula for the Vakhitov-Kolokolov condition and the sole difficulty is the \emph{explicit} evaluation of such an expression. Nevertheless, this already allows for a numerical verification of the Vakhitov-Kolokolov condition for each $p$ and $\lambda$, resulting in Figure \ref{fig:enter-label}.

\begin{figure}[h]
    \centering
    \includegraphics[height=6cm]{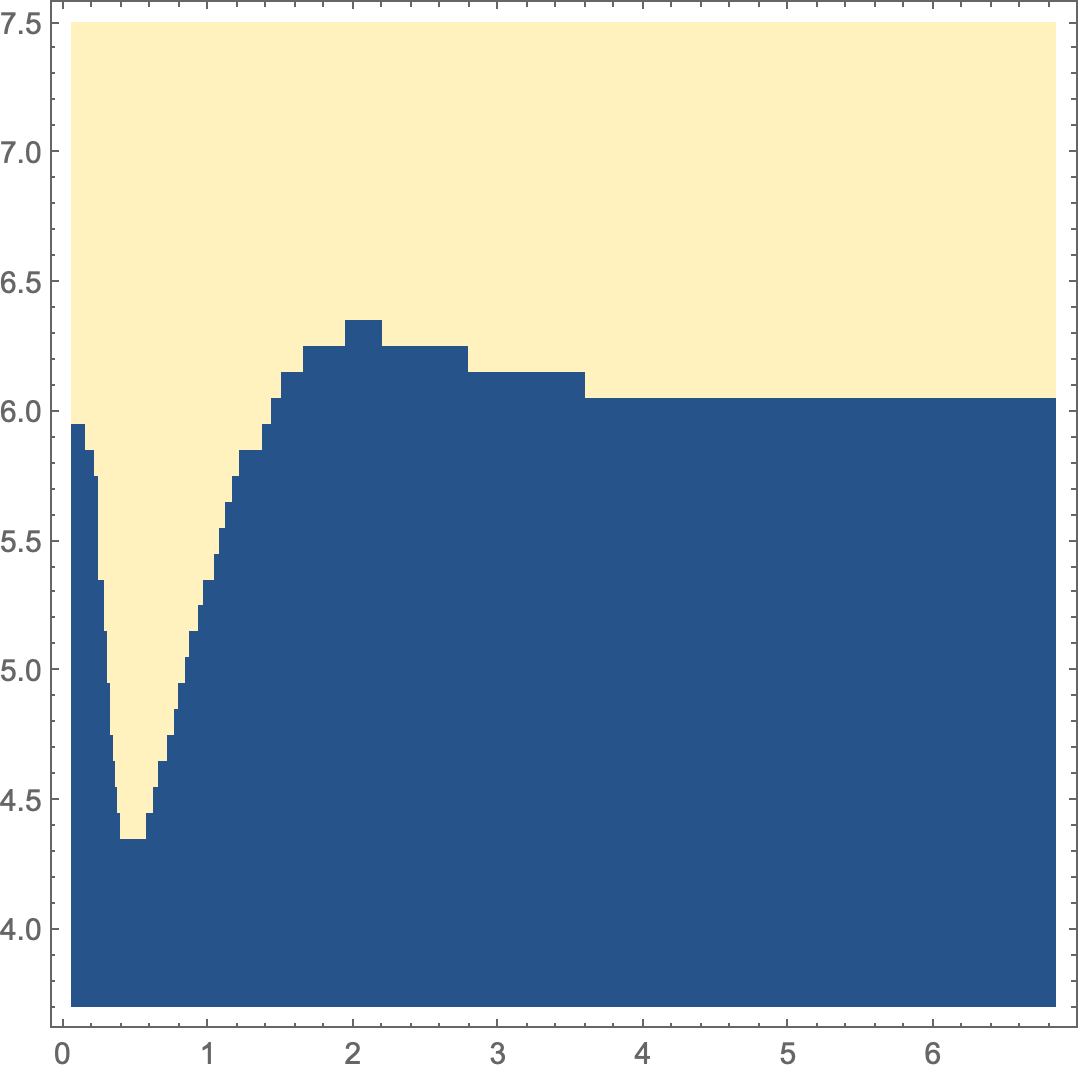}\quad\quad \includegraphics[height=6cm]{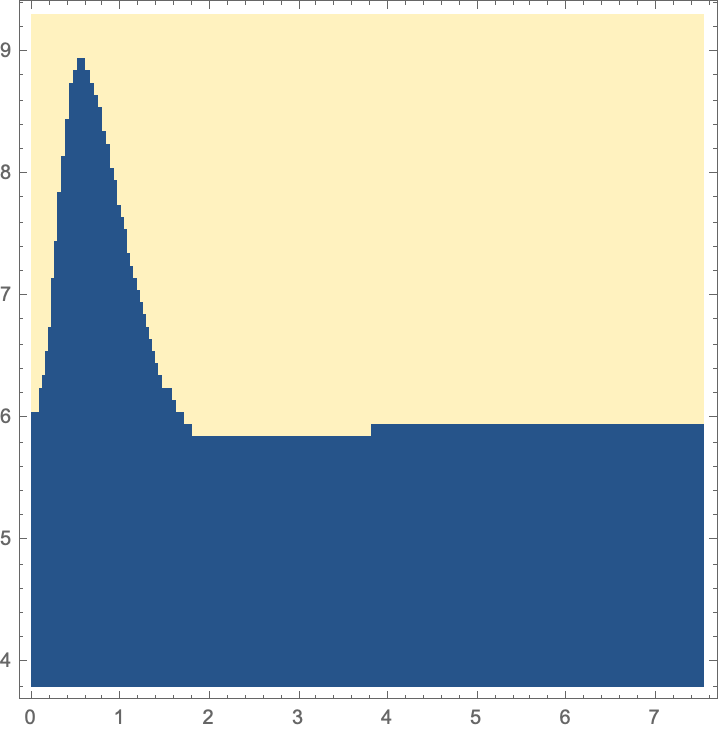}
    \caption{Numerical simulation of the stable (blue) and unstable (yellow) regions in the $(\lambda,p)$ plane, for the  $\mathcal{T}$-graph (left) and the tadpole graph (right).}
    \label{fig:enter-label}
\end{figure}

    Even though our theoretical results validate the above numerical computations, it would be interesting to derive Figure \ref{fig:enter-label} analytically.
\end{remark}

\begin{remark}
 In this work, we have constrained ourselves to the $\mathcal{T}$ and tadpole graphs. We conjecture that our results can be adapted to more general classes of metric graphs, such as single-knot metric graphs \cite{AgostinhoCorreiaTavares2}, which consist of arbitrary numbers of half-lines,  pendants and loops, all attached to the same vertex. In the regular case (where pendants have length $\ell$ and loops have length $2\ell$), action ground-states can be analyzed as in the  $\mathcal{T}$ or tadpole graphs (see \cite[Section 2]{AgostinhoCorreiaTavares2}). The main obstacles to extend our results to such graphs are the nondegeneracy condition imposed by the Grillakis-Shatah-Strauss theory (see Section \ref{AssumptionA_3}) {and the non-monotonicity of the map \eqref{eq:mass_function} for $p=6$ (see Section \ref{sec:NonMonotoneTheta})}. We believe that this is simply a technical difficulty and that stability transitions hold for regular single-knot (or even generic) metric graphs. {For this reason, whenever possible, we keep the discussion as general as possible (see also Remark \ref{rem:theta} below).}
\end{remark}

\bigskip
{
The paper is structured as follows. Section \ref{sec:GSS} concerns the application of the well-known Grillakis-Shatah-Stauss theory in the context of (NLS). In particular, we prove a general criterion for the existence of stability transitions, see Theorem \ref{thm:transition}.  Section \ref{sec:Preliminaries} includes an overview of known results about action ground-states, mostly taken from \cite{AgostinhoCorreiaTavares,AgostinhoCorreiaTavares2}, including explicit expressions for the mass and the length of the compact edge in terms of the phase-plane variables.

After these preliminary sections, in Section \ref{sec:GSS_particular}, we verify the assumptions of the Grillakis-Shatah-Strauss theory on the $\mathcal{T}$ and the tadpole graphs. Section \ref{sec:NonMonotoneTheta} is then dedicated to the proof of the lack of monotonicity (in $\lambda$) of the mass of action ground-states for $p=6$, which allows to apply the aforementioned general criterion and deduce stability transitions close to $p=6$. This  concludes the proof of Theorem \ref{thm:transition_particular}.

The final Section \ref{sec:asymptotics} is dedicated to the proof of Theorem \ref{thm:asymptotics}, which concerns the asymptotic cases $\lambda\sim 0,\infty$ (Section \ref{sec:asymptotics_lambda}) and $p\sim 2^+,\infty$ (Section \ref{sec:asymptotics_p}).
}

\section{Grillakis-Shatah-Strauss theory}\label{sec:GSS}

As mentioned in the introduction, the study of orbital stability of action ground-states of \eqref{NLSE_Chapter5} fits within the framework of the Grillakis-Shatah-Strauss stability theory. This theory, albeit sufficiently general to analyze a very large class of equations, stands on three main assumptions which may be, in some situations, hard to verify. In this section, we introduce some notation in order to connect with the results in \cite{grillakis1987stability}, and we conclude with a general criterion that ensures stability transitions.

 For now, we keep the discussion as general as possible, so we work with $\mathcal{G}$ a general metric graph with finitely many edges. We consider $X=H^1(\mathcal{G};\C)$ to be the space of complex valued functions which are $H^1$ in each edge and continuous over the graph $\mathcal{G}$. We identify $X\simeq H^1(\mathcal{G}; \R)\times H^1(\mathcal{G};\R)$ and we endow $X$ with the usual real inner product. Moreover, we denote by $X'$ the dual space of $X$ and use the notation $\langle\cdot,\cdot\rangle$ for the duality pairing between $X'$ and $X$.

Recall from \eqref{eq:actionfunctional} the $C^2$-action functional $S_\lambda$. It can be decomposed into two functionals of class that are formally preserved under the flow of \eqref{NLSE_Chapter5}:
\begin{equation*}
	S_\lambda(u,\mathcal{G})=E(u,\mathcal{G})+\frac{\lambda}{2} Q(u,\mathcal{G}),
\end{equation*}
where $E,Q:H^1(\mathcal{G};\C)\to\R$ are given by
\begin{equation*}
	E(u,\mathcal{G})=\frac{1}{2}\int_{\mathcal{G}}|u'|^2dx-\frac{1}{p}\int_{\mathcal{G}}|u|^pdx\quad\text{(\textit{Energy})}\quad\text{and}\quad Q(u,\mathcal{G})=\int_{\mathcal{G}}|u|^2dx\quad\text{(\textit{Mass})}.
\end{equation*}

With this notation, equation \eqref{NLSE_Chapter5} can be written as an abstract Hamiltonian system of the form 
$$\partial_t v = JE'(v(t),\mathcal{H}_1),$$
where $J:X'\to X'$ is the closed skew-symmetric linear operator which is represented by
$$ J=\begin{pmatrix}
	0 & 1\\-1 & 0
\end{pmatrix}.$$
For more details, we refer the reader to \cite[Section 6]{grillakis1987stability}. In this abstract setting, there are three assumptions that need to be verified in order to apply the Grillakis-Shatah-Strauss stability theory. In our setting, they can be defined as follows: 

\begin{enumerate}
	\item[$(A_1)$](Well Posedeness) Given $v_0\in X$, there exists a maximal time of existence $T_0>0$ and a unique solution $v\in C([0,T_0),X)$ to \eqref{NLSE_Chapter5} such that $v(0)=v_0$,
	$E\left(v(t),\mathcal{G}\right)=E\left(v_0,\mathcal{G}\right)$ and $Q\left(v(t),\mathcal{G}\right)=Q\left(v_0,\mathcal{G}\right)$, for all $t\in[0,T_0)$.
	\item[$(A_2)$](Existence and regularity of a bound-state curve) There exist $0<\lambda_1<\lambda_2$ and a $C^1$ map $\lambda\mapsto \Phi_\lambda$ from the open interval $(\lambda_1,\lambda_2)$ into $X$ that satisfies:
	\begin{equation}\label{CriticalPointAction}
		\langle S_\lambda'(\Phi_\lambda,\mathcal{G}),v\rangle=\langle E'(\Phi_\lambda,\mathcal{G})+\frac{\lambda}{2} Q'(\Phi_\lambda,\mathcal{G}),v\rangle=0,\ \text{for all}\ v\in H^1(\mathcal{G};\C).
	\end{equation}
	
	\item[$(A_3)$](Spectral Conditions) Define the linearized Hamiltonian around the bound-state $\Phi_\lambda\in X$ by $H_\lambda: X\to X'$, which satisfies
	\begin{equation}\label{LinearizedOperator}
		\langle H_\lambda v,w\rangle:=\langle S_\lambda''(\Phi_\lambda,\mathcal{G})v,w\rangle.
	\end{equation}
	Assume that $H_\lambda$
	satisfies the following spectral properties: for each $\lambda\in(\lambda_1,\lambda_2)$, $H_\lambda$ has exactly one negative simple eigenvalue, its kernel is spanned by $i\Phi_\lambda$, and the rest of its spectrum is positive and bounded away from zero.
\end{enumerate}

\smallbreak

The Grillakis-Shatah-Strauss stability theory tells us that, under these assumptions, we can discuss the stability of the bound-state $\Phi_\lambda$ for each $\lambda>0$, in terms of the convexity of the function $d:(0,\infty)\to\R$ defined by 
$$ 
d(\lambda)=E(\Phi_\lambda,\mathcal{G})+\frac{\lambda}{2} Q(\Phi_\lambda,\mathcal{G})=S_\lambda(\Phi_\lambda,\mathcal{G}).
$$
This reduces, see \cite[Theorem 3]{grillakis1987stability}, to analyzing the sign of $d''(\lambda)$. Observe that if $\Phi_\lambda$ is a solution of the equation \eqref{StationaryNLS_Chapter5}, we have that
$$ 
d''(\lambda)=\frac{1}{2}\frac{d }{d\lambda}Q(\Phi_\lambda,\mathcal{G}).
$$
Indeed, by definition of $d$, and assumption $(A_2)$, we have
\begin{align*}
	d'(\lambda)&=\left\langle E'(\Phi_\lambda,\mathcal{G}),\frac{d}{d\lambda}\Phi_\lambda \right\rangle+\frac{1}{2}Q(\Phi_\lambda,\mathcal{G})+\frac{\lambda}{2}\left\langle Q'(\Phi_\lambda,\mathcal{G}),\frac{d}{d\lambda}\Phi_\lambda\right\rangle=\frac{1}{2}Q(\Phi_\lambda,\mathcal{G}).
\end{align*}

Then, stability can be determined as follows:
\begin{theorem}[\textit{Vakhitov--Kolokolov Stability Criterion \cite{grillakis1987stability}}]\label{th:5.27}
	Fix $p>2$. Suppose assumptions $(A_1)$, $(A_2)$ and $(A_3)$ hold. Then,
	\begin{enumerate}
		\item If $\frac{d Q}{d \lambda}(\Phi_\lambda,\mathcal{G})>0$, the bound-state $\Phi_\lambda$ is orbitaly \textit{stable};
		\item If $\frac{d Q}{d \lambda}(\Phi_\lambda,\mathcal{G})<0$, the bound-state $\Phi_\lambda$ is orbitaly \textit{unstable}.
	\end{enumerate}
\end{theorem}

\subsection{A general criterion for stability transitions}

In order to present a general criterion for stability transitions close to $p=6$, we need to assume some upper and lower estimates of the size of action ground-states. The validity of this condition is usually ensured by using the variational characterization of action ground-states.
\begin{itemize}
\item[$(H)$] Let $\Phi_\lambda$ be as in $(A_2)$, satisfied for $\lambda\in (0,\infty)$, and consider, for $\lambda=\ell^2>0$, the scaling $$\Phi_\lambda(x)=\ell^\frac{2}{p-2}\Phi_\ell(\ell x)=\lambda^\frac{1}{p-2}\Phi_\ell(\sqrt{\lambda}x),\quad x\in \mathcal{G_\lambda}.$$  Then there exists $C_p>0$ such that
\begin{equation}\label{lower_upper_bounds}
\frac{1}{C_p}\leq \|\Phi_\ell\|_{H^1(\mathcal{G}_\lambda)}\leq C_p\qquad \text{ for all } \ell>0.
\end{equation}
\end{itemize}
\begin{lemma}\label{5.49.Temp}
	Let $\G$ be a non-compact metric graph and assume $(A_2)$ and $(H)$ for some $p>2$. Then there exists $C=C_p>0$ such that 
	\begin{equation}\label{eqn.lema3.14}
		\frac{1}{C}\lambda^{\frac{6-p}{2(p-2)}}\leq Q(\Phi_\lambda,\mathcal{G})\leq C\lambda^{\frac{6-p}{2(p-2)}}.
	\end{equation}
\end{lemma}
\begin{proof}
From the scaling and assumption $(H)$, 
	\begin{equation*}\label{eqn.lema3.14.1}
		Q(\Phi_\lambda,\mathcal{G})\lambda^{\frac{p-6}{2(p-2)}}=\|\Phi_\ell\|^2_{L^2(\G_\lambda)}\leq \|\Phi_\ell\|^2_{H^1(\G_\lambda)} \leq C_p^2
	\end{equation*}
	Recalling the Gagliardo-Nirenberg inequality
	\[
\|\Phi_\ell\|_{H^1(\mathcal{G}_\lambda)}^2=\|\Phi_\ell\|_{L^p(\mathcal{G}_\lambda)}^p\leq C \|\Phi_\ell\|^\frac{p+2}{2}_{L^2(\mathcal{G}_\lambda)}\|\Phi_\ell'\|^\frac{p-2}{2}_{L^2(\mathcal{G}_\lambda)}\leq \kappa C \|\Phi_\ell\|^\frac{p+2}{2}_{L^2(\mathcal{G}_\lambda)},
	\]
	the remaining conclusion follows from the lower bounds in $(H)$. 
\end{proof}

\begin{theorem}\label{thm:transition} 
    	Suppose that $\G$ is a metric graph. For $p\sim 6$, assume conditions $(A_1)$, $(A_2)$ and $(A_3)$ are satisfied for $\lambda\in (0,\infty)$, and that $(H)$ holds. For $\Theta(p,\lambda):=Q(\Phi_\lambda,\mathcal{G})$, suppose that the map $\lambda\in (0,\infty) \mapsto\Theta(6,\Phi_\lambda)$ is not monotone and that $(p,\lambda)\mapsto\frac{\partial\Theta}{\partial\lambda}(p,\lambda)$ is continuous. Then, there exists $\varepsilon>0$ such that
	\begin{enumerate}
		\item for $p\in(6-\varepsilon,6)$, there exists a SUS-stability transition.
		\item  for $p\in(6,6+\varepsilon)$, there exists a USU-stability transition.
	\end{enumerate} 
\end{theorem}

\begin{proof}[Proof of Theorem \ref{thm:transition}] 
    Under the assumptions of the statement, we can apply the Grillakis-Shatah-Strauss stability theory to the curve  $\lambda\in (0,\infty)\mapsto \Phi_\lambda$. The only thing left to do is to study the Vakhitov--Kolokolov criterion, recall Theorem \ref{th:5.27}. 
    
	Since $\lambda\mapsto\Theta(6,\lambda)$ is not monotone, assume, without loss of generality, there exist $0<\lambda_1<\lambda_2$ for which
	\[
    \frac{\partial\Theta}{\partial\lambda}(6,\lambda_1)<0\quad\text{and}\quad\frac{\partial\Theta}{\partial\lambda}(6,\lambda_2)>0
    \]
	(the situation where $\lambda_2<\lambda_1$ is analogous). Since, $(p,\lambda)\mapsto\frac{\partial\Theta}{\partial\lambda}(p,\lambda)$ is continuous, for $\varepsilon>0$ sufficiently small, 
	\begin{align*}
		&\frac{\partial\Theta}{\partial\lambda}(p,\lambda)<0\quad \text{for}\ p\in [6-\varepsilon,6+\varepsilon],\ \lambda\in [\lambda_1-\varepsilon,\lambda_1+\varepsilon],\\
        &\frac{\partial\Theta}{\partial\lambda}(p,\lambda)>0\quad \text{for}\ p\in [6-\varepsilon,6+\varepsilon],\ \lambda\in [\lambda_2-\varepsilon,\lambda_2+\varepsilon].
	\end{align*}
	Now, in the supercritical case $p\in (6,6+\varepsilon]$, since  $\Theta(p,\lambda)\to0^+$ as $\lambda\to+\infty$ (by Lemma \ref{5.49.Temp}),  there exist $\lambda_3,\lambda_4$ with $\lambda_2+\varepsilon<\lambda_3<\lambda_4$ such that 
	\begin{equation}\label{GSSAux.2}
		\frac{\partial\Theta}{\partial\lambda}(p,\lambda)<0\quad \text{for}\ \lambda\in(\lambda_3,\lambda_4),
	\end{equation}
	which concludes the proof of Item 2. In the subcritical case $p\in [6-\varepsilon,6)$,  we use once again Lemma \ref{5.49.Temp} to ensure the existence of $\lambda_5,\lambda_6$ with  $\lambda_5<\lambda_6<\lambda_1-\varepsilon$ for which
	\begin{equation}\label{GSSAux.2_Sub}
		\frac{\partial\Theta}{\partial\lambda}(p,\lambda)>0\quad \text{for}\ \lambda\in(\lambda_5,\lambda_6)
	\end{equation}
    and the proof is finished.
\end{proof}

\begin{remark}\label{rem:lambda_estrela}
    In the case of the $\mathcal{T}$ and tadpole graphs, the spectral condition $(A_3)$ will be verified for all $\lambda$ except for a problematic frequency $\lambda^*=\lambda^*(\mathcal{G})$ (see Section \ref{AssumptionA_3}). This poses no particular obstacle for the validity of the above result, as one can choose the various frequency intervals so that they do not include $\lambda^*$.
\end{remark}

\section{Review of known results for the $\mathcal{T}$ and tadpole graphs}\label{sec:Preliminaries}

In this section, we review known results on the qualitative and quantitative properties of action ground-states on the $\mathcal{T}$ and tadpole graphs. 
Unless we want to stress that a result or a proof holds only for a specific graph, we shall denote the graphs $\mathcal{T}_\ell$ and $\G_\ell$ by $\mathcal{H}_\ell$. 

From now on, $u^\lambda$ will denote the positive action ground-state solution on $\mathcal{H}_1$ (we omit here the dependence on $p$ of simplicity). Since we will study the mass of $u^\lambda$ as a function of $p$, we denote
\[
(p,\lambda)\mapsto \Theta(p,\lambda)=\|u^\lambda\|_{L^2(\mathcal{H}_1)}^2.
\]
A key element in our analysis is the function $f:\R^+_0\to \R$ defined by $f(t)=\frac{t^2}{2}-\frac{t^p}{p}$ (see Figure \ref{fig:f}). We write $f_1=f|_{[0,1]}$ and $f_2=f|_{[1,\infty)}$ and observe that both functions are invertible over the respective domains.

\begin{figure}[h]\label{fig:f}
\begin{center}
\includegraphics[height=3cm]{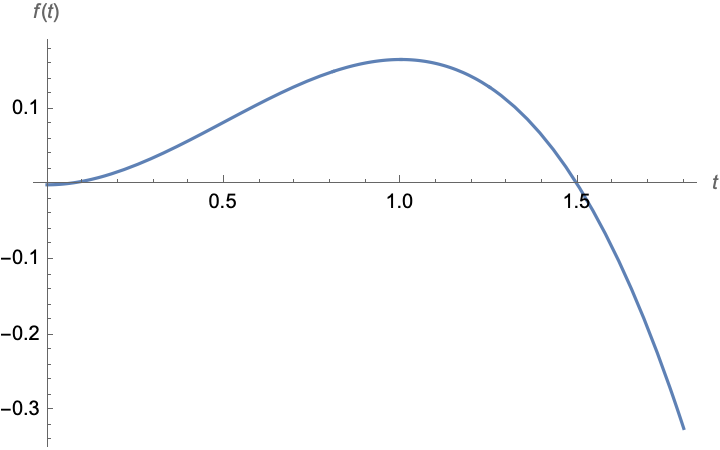}
\end{center}
\caption{Plot of the function $f(t)$ for $p=3$.}
\end{figure}

Moreover, we denote by $\varphi$ the soliton (for $\lambda=1$) centered at zero in the real line, i.e., the unique positive solution of 
\[
-\varphi''+\varphi=\varphi^{p} \text{ in } \R,
\] 
symmetric (and decreasing) with respect to the origin. A direct integration of the equation shows that
\[
\varphi(x)=\left(\frac{p}{2}\right)^\frac{1}{p-2}\sech^{\frac{2}{p-2}}\left(\frac{p-2}{2}x\right),\quad x\in\R,
\]
from which we observe that $\varphi(0)=(p/2)^\frac{1}{p-2}$.

\bigskip

We recall the following uniqueness result.

\begin{theorem}[{\cite[Theorem 1.3]{AgostinhoCorreiaTavares} $\&$ \cite[Proposition 1.17-1]{AgostinhoCorreiaTavares2}}]\label{thm:uniqueness}
For every $\lambda>0$, there exists a unique positive action ground-state solution $u^\lambda$ of \eqref{StationaryNLS_Chapter5}. Moreover:
\begin{enumerate}
\item In the $\mathcal{T}$-graph, $u^\lambda$ is symmetric and stricly decreasing in the half-lines, being increasing in the pendant, with a global maximum at the tip.
\item In the tadpole graph, $u^\lambda$ is strictly decreasing in the half-line; it is  symmetric in the loop with respect to the middle point $x=1$, being strictly increasing in $[0,1]$.
\end{enumerate}
\end{theorem}

To state further properties, we work with the rescaled ground-state $u_\ell(x)=\ell^\frac{2}{p-2}u^{\lambda}(\ell x)=\lambda^\frac{1}{p-2}u^{\lambda}(\sqrt{\lambda}x)$, where $\ell=\sqrt{\lambda}$, which solves
\[
-u_\ell''+u_\ell=|u_\ell|^{p-1}u_\ell \text{ in } \mathcal{H}_\ell.
\]
For both graphs, we denote by $u_h=u_{h,\ell}$ the solution on the half-lines, which are identified with $[0,\infty)$, and $u_\kappa=u_{\kappa,\ell}$ the solution on the pendant (on $\mathcal{T}_\ell$) and on the half-loop (on $\mathcal{G}_\ell$), identified with the interval $[0,\ell]$).  We have the following result, whose proof can be found in \cite{AgostinhoCorreiaTavares,AgostinhoCorreiaTavares2} (more precisely, by combining the uniqueness statements written above with Lemma 2.1 of the second paper).
\begin{theorem}\label{thm:description_edge_by_edge}
Under the previous notation, there exists a \emph{unique} $y=y(p,\ell)>0$ such that
\begin{enumerate}
\item for each half-line $h$, $u_h(x)=\varphi(x+y)$;
\item defining $z=z(p,\ell):=\varphi(y)=\varphi(-y) \in\left(0,(p / 2)^{1 /(p-2)}\right)$, on the compact edge identified with $[0,\ell]$, $u_\kappa$ solves the overdetermined problem
\begin{equation}\label{eq:sol_finite_edge}
-u_\kappa''+u_\kappa=u_\kappa^{p-1} \text { in }[0, \ell], \quad u_\kappa(0)=z, \quad u_\kappa^{\prime}(0)=\theta \sqrt{2 f(z)}, \quad u_\kappa^{\prime}(\ell)=0,
\end{equation}
where $\theta=2$ in the $\mathcal{T}$-graph $\mathcal{T}_\ell$, and $\theta=1/2$ in the tadpole graph $\mathcal{G}_\ell$. In particular,
\begin{equation}\label{eq:hamilt}
    -\frac{1}{2}(u_\kappa')^2+f(u_\kappa)=(1-\theta^2)f(z)
\end{equation}
\end{enumerate}
Moreover,  $z(p,\ell)$ is uniquely determined by the identity
\[
 L(p,z(p,\ell))=\ell,
\]
where, for each $\theta>0$ and $p\in(2,\infty)$, $L(p,\cdot):\left(0,\left(\frac{p}{2}\right)^{\frac{1}{p-2}}\right) \rightarrow \mathbb{R}^{+}$ is the length function 
\begin{equation}\label{eq:Length}
L(p,z)=\frac{1}{\sqrt{2}} \int_z^{f_{2}^{-1}\left(\left(1-\theta^2\right) f(z)\right)} \frac{d t}{\sqrt{f(t)-\left(1-\theta^2\right) f(z)}}.
\end{equation}
Finally, assumption (H) holds, that is, there exists a constant $C_p>0$ for which
$$
\frac{1}{C_p}\leq \|u_\ell\|_{H^1(\mathcal{G})}\leq C_p.
$$
\end{theorem}
The length function was introduced in \cite[Definition 2.7]{AgostinhoCorreiaTavares2} (see also \cite[Definition 2.2]{AgostinhoCorreiaTavares} for the $\mathcal{T}$-graph case), corresponding to an explicit expression of  the length of the interval for increasing solutions of \eqref{eq:sol_finite_edge} written only in terms of the initial data. From \cite[Proposition 2.8]{AgostinhoCorreiaTavares2}, we have that
\begin{equation}\label{eq:asymptotic_L}
\lim_{z\to 0^+} L(p,z)=+\infty\quad \text{ and } \quad \lim_{z\to (p/2)^\frac{1}{p-2}} L(p,z)=0
\end{equation}
whenever $\theta>0$, being strictly decreasing for $\theta\in (0,2]$ (which includes the case of the $\mathcal{T}$ and tadpole graphs). 
{
\begin{remark} \label{rem:theta} Even though, for the purpose of this paper, only the cases $\theta=2,\frac{1}{2}$ matter, having in mind future extensions to other graphs, we keep in many statements an arbitrary $\theta>0$, whenever the result is true in such general setting. In the class of regular single-knot graphs introduced in \cite{AgostinhoCorreiaTavares2},  consisting of graphs with $H$ half-lines, $P$ pendants and $L$ loops, all attached to the same vertex, the solutions on the compact edges satisfy \eqref{eq:sol_finite_edge}, for an \emph{incidence index related to} $H$, $L$, $P$ and the monotonicity of the solution on the half-lines. For more details, we refer to \cite{AgostinhoCorreiaTavares2}, in particular to equation (1.8) therein.
\end{remark}
}

For the purpose of studying orbital stability (in particular in applying the {Vakhitov--Kolokolov stability criterion}, see the introduction or Theorem \ref{th:5.27} above), we also need to recall the following. 

\begin{lemma}  \label{lemma:L'}
	For $\theta>0$ and $z\in \left(0,(p/2)^{1/(p-2)}\right)$,
	\begin{align}
		\sqrt{2}L(p,z)\left((1-\theta^2)\right.&f(z)-\left.f(1)\right)=-2\theta\frac{(f(z)-f(1))\sqrt{f(z)}}{f'(z)} \nonumber \\
		&-\int_z^{f_2^{-1}\left((1-\theta^2)f(z)\right)}\left(3-2\frac{(f(t)-f(1))}{f'(t)^2}f''(t)\right)\sqrt{f(t)-(1-\theta^2)f(z)}dt\label{eqn.LDiffForm1}
	\end{align}
	and, taking derivatives in $z$,
	\begin{align}
&\sqrt{2}\frac{\partial L}{\partial z}(p,z)\left((1-\theta^2)f(z)-f(1)\right)=\theta\frac{f(1)}{\sqrt{f(z)}}+(1-\theta^2)f'(z)\int_z^{f_{2}^{-1}\left((1-\theta^2)f(z)\right)}\frac{A(t)dt}{\sqrt{f(t)-(1-\theta^2)f(z)}}\label{H_p(z,theta)NotSingular}\\
		&=\frac{\theta\left(f(1)-2(1-\theta^2)A(z)f(z)\right)}{\sqrt{f(z)}}-(1-\theta^2)f'(z)\int_z^{f_{2}^{-1}\left((1-\theta^2)f(z)\right)}\frac{\rho(t)}{f'(t)^4}\sqrt{f(t)-(1-\theta^2)f(z)}dt,\noindent
	\end{align}
	where:
	\begin{itemize}
	\item $A:\R^+\to \R$ is defined by
	\begin{equation*}
A(t)=\frac{1}{2}-\frac{(f(t)-f(1))f''(t)}{{f'}^2(t)},
	\end{equation*}
	having a continuous extension to $\mathbb{R}^{+}$ (still denoted by $A$), with the following properties: 
    \[
    A(t)>0 \text{ for } t \in(0,1),\ A(t)<0 \text{ for } t \in(1,+\infty),\text{ and } A(1)=0. 
    \]
    Moreover, $\lim_{t\to 0^+} t^2A(t)=f(1)$ and $\lim _{t \rightarrow \infty} A_p(t)=-\frac{p-2}{2 p}<0$.
	\item $\rho:\R^+\setminus\{1\}\to\R$ is defined by 
    \[
    \rho(t)=-3f''(t)f'(t)^2+6(f(t)-f(1))f''(t)^2-2(f(t)-f(1))f'''(t)f'(t),
    \]
    and satisfies $\rho(t)<0$ for $t\neq 1$, and $\lim_{t\to 0^+}\rho(t)=-6f(1)$.
	\end{itemize}

    \smallbreak
	 In particular, the maps $(p,z)\mapsto L(p,z),\ \frac{\partial L}{\partial z}(p,z)$ are continuous. Finally, for $\theta\in (0,2]$, $\frac{\partial L}{\partial z}(p,z)<0$.
\end{lemma}
\begin{proof}
The identities involving $L$ and its derivative correspond to Lemma 2.10, Lemma 2.11 and Proposition 2.12 in \cite{AgostinhoCorreiaTavares2}. Regarding the continuity of $L(p,z)$ in $p$ and $z$, as we are away from $t=0$ and the integrand is uniformly bounded, the right hand side of \eqref{eqn.LDiffForm1} is continuous in both variables. Since $(1-\theta^2)f(z)-f(1)$ does not vanish (it is negative), it follows that $(p,z)\mapsto L(p,z)$ is continuous. 
	
As for the continuity of $\frac{\partial L}{\partial z}(p,z)$, again from the fact that $(1-\theta^2)f(z)-f(1)<0$, we just need to check that the right-hand-side of \eqref{H_p(z,theta)NotSingular} is continuous;	for that,  it is enough to show that $\lim\limits_{t\to1}\frac{\rho(t)}{f'(t)^4}$ exists and is finite. But
	$$\lim_{t\to1}\frac{\rho(t)}{f'(t)^4}=\lim_{t\to1}-\frac{(p-2)^3(p^2+p-2)(t-1)^4+o((t-1)^5)}{(p-2)^4(t-1)^4+o((t-1)^5)}=-\frac{p^2+p-2}{p-2}, $$
	which concludes the proof. Finally, the fact that $\frac{\partial L}{\partial z}(p, z)<0$, although not explicitly stated in  \cite[Proposition 2.8]{AgostinhoCorreiaTavares2} (which only mentions strict monotonicity of $L$), is included in its proof. 
\end{proof}

With the aid of this result, we are able to study the variation of $y(p,\ell)$ and $z(p,\ell)$ with respect to $\ell$.

\begin{lemma}\label{cor:4.2.Temp}
	 The map $(p,\ell)\mapsto (y(p,\ell),z(p,\ell))$ is continuous, $C^1$ in $\ell$, with\footnote{$L^{-1}(p,\ell)$ denotes the inverse of $L$ in the $z$ variable, for each fixed $p$.} \begin{equation}\label{derz(p,ell)}
		\frac{\partial z}{\partial \ell}(p,\ell)=\frac{1}{\frac{\partial L}{\partial z}(L^{-1}(p,\ell))}<0,
	\end{equation}
	\begin{equation}\label{dery(p,ell)}
		\frac{\partial y}{\partial \ell}(p,\ell)=\frac{1}{\varphi'(y(p,\ell))}\frac{1}{\frac{\partial L}{\partial z}(L^{-1}(p,\ell))}>0.
	\end{equation}
\end{lemma}
\begin{proof}
	Since $\frac{\partial L}{\partial z}(p,z)<0$ and $\varphi'(x)<0$ for $x>0$, by the implicit function theorem, the maps $L^{-1}$ and $\varphi^{-1}$ are continuous in $p$ and $C^1$ in $\ell$. Therefore, the claimed smoothness of $y$ and $z$ follows from the identities
    $$z(p,\ell)=L^{-1}(p,\ell),\qquad y(p,\ell)=\varphi^{-1}(z(p,\ell)).$$ 
    The identities in \eqref{derz(p,ell)} and \eqref{dery(p,ell)} are obtained by direct differentiation with respect to $\ell$. The inequality in \eqref{derz(p,ell)} follows from Lemma \ref{lemma:L'}, and the inequality in  \eqref{dery(p,ell)} follows from Lemma \ref{lemma:L'} and the fact that $y(p,\ell)>0$.
\end{proof}

\subsection{Explicit expressions for the mass of action ground-states}\label{sec:explicit_mass}

In the previous paragraphs, we recalled an explicit expression for the length $\ell$ to the solutions of the overdetermined problem \eqref{eq:sol_finite_edge}. Here, for future purposes, we do the same for the mass of action ground-states, which we always denote (in both graphs) by
\[
(p,\lambda)\mapsto \Theta(p,\lambda):=\|u^\lambda\|_{L^2(\mathcal{H}_1)}^2.
\]
By scaling,  we have
\begin{equation}\label{ThetaScaled}
	\Theta(p,\lambda)=\lambda^\frac{6-p}{2(p-2)}\|u_\ell\|^2_{\mathcal{H}_\ell},
\end{equation} 
where $\lambda=\ell^2$. From now on, we denote by $\Theta_1:(2,\infty)\times(0,\infty)\mapsto\R^+$ the function defined by 
\begin{equation}\label{eq:Theta_1}
\Theta_1(p,\ell):=\norm{u_\ell}{2}{\mathcal{H}_\ell}^2.
\end{equation}

\begin{lemma} \label{lemma:Theta_1_expressions}
For each $p>2$, $z\in\left(0,\left(\frac{p}{2}\right)^\frac{1}{p-2}\right)$, and $\theta>0$, define $\mu_1(p,z,\theta)$ and $\mu_2(p,z)$ by
\begin{equation}\label{mu_functions}
    \mu_1(p,z,\theta)=\frac{1}{\sqrt{2}}\int_{z}^{f_{2}^{-1}\left((1-\theta^2)f(z)\right)}\frac{t^2dt}{\sqrt{f(t)-(1-\theta^2)f(z)}}\quad\text{and}\quad\mu_2(p,z)=\frac{1}{\sqrt{2}}\int_0^z\frac{t^2dt}{\sqrt{f(t)}}.
\end{equation}
Then:
\begin{enumerate}
\item In the $\mathcal{T}$-graph $\mathcal{T}_\ell$, 
\begin{equation}\label{Theta_1_T'}
	\Theta_1(p,\ell)=\Theta_1(p,z(p,\ell))=\mu_1\left(p,z(p,\ell),2\right)+2\mu_2(p,z(p,\ell)).
\end{equation}
\item In the tadpole graph $\mathcal{G}_1$,
\begin{equation}\label{Theta_1_TADPOLE'}
    \Theta_1(p,\ell)=\Theta_1(p,z(p,\ell))=2\mu_1\left(p,z(p,\ell),\frac{1}{2}\right)+\mu_2(p,z(p,\ell)).
\end{equation}
\end{enumerate}
\end{lemma}
\begin{proof}
By Theorem \ref{thm:description_edge_by_edge}, using, in each edge, the change of variable $t=u(x)$ (which is possible, by the monotonicity of the solutions), we can rewrite the function $\Theta_1$ in $\mathcal{T}_\ell$ as
\begin{align}\label{Theta_1_T}
	\Theta_1(p,\ell)&=\int_0^\ell u^2_{\ell,h}(x)dx+2\int_0^\infty u^2_{\ell,\kappa}(x)dx\nonumber\\
	&=\frac{1}{\sqrt{2}}\int_{z(p,\ell)}^{f_{2}^{-1}\left(-3f(z(p,\ell))\right)}\frac{t^2dt}{\sqrt{f(t)+3f(z(p,\ell))}}+\frac{2}{\sqrt{2}}\int_0^{z(p,\ell)}\frac{t^2}{\sqrt{f(t)}}dt
\end{align}
Similarly, in the tadpole $\G_\ell$,
\begin{align}\label{Theta_1_TADPOLE}
	\Theta_1(p,\ell)&=2\int_0^\ell u^2_{\ell,\kappa}(x)dx+\int_0^\infty u^2_{\ell,h}(x)dx\nonumber\\
	&=\frac{2}{\sqrt{2}}\int_{z(p,\ell)}^{f_{2}^{-1}\left(\frac{3}{4}f(z(p,\ell))\right)}\frac{t^2dt}{\sqrt{f(t)-\frac{3}{4}f(z(p,\ell))}}+\frac{1}{\sqrt{2}}\int_0^{z(p,\ell)}\frac{t^2}{\sqrt{f(t)}}dt.
\end{align}\qedhere
\end{proof}

\begin{lemma}[Regularity and expressions of $\frac{\partial \mu_1}{\partial z}$ and $\frac{\partial \mu_2}{\partial z}$]\label{lemma:mu_1reg}
	Fix $\theta\in(0,\infty)$ and $0< z<\left(\frac{p}{2}\right)^\frac{1}{p-2}$. We have
    \begin{equation}\label{mu2DerAux}
		\sqrt{2}\frac{\partial \mu_2}{\partial z}(p,z)=\frac{z^2}{\sqrt{f(z)}}.
	\end{equation}
    Moreover,
    \begin{align}
		&\sqrt{2}\mu_1(p,z,\theta) \left((1-\theta^2)f(z)-f(1)\right)=-2\theta z^2\frac{(f(z)-f(1))}{f'(z)}\sqrt{f(z)}\nonumber\\
		-&\int_z^{f_{2}^{-1} \left((1-\theta^2)f(z)\right)}g(t)\sqrt{f(t)-(1-\theta^2)f(z)}dt,\label{eqn.4.21}
	\end{align}
    and
    \begin{align}
		\sqrt{2}\frac{\partial\mu_1}{\partial z}(p,z,\theta)&((1-\theta^2)f(z)-f(1))=\theta z^2\frac{f(1)}{\sqrt{f(z)}}\nonumber\\
        &-(1-\theta^2)f'(z)\int_z^{f_{2}^{-1}((1-\theta^2)f(z))}\left(t^2-\frac{g(t)}{2}\right)\frac{dt}{\sqrt{f(t)-(1-\theta^2)f(z)}},\label{mu1DerAux}
	\end{align}
	where, for every $p>2$, the function $g:\R\to\R$ is given by
	$$g(t)=3t^2-2t^2\frac{(f(t)-f(1))}{f'(t)^2}f''(t)+4t\frac{(f(t)-f(1))}{f'(t)}.
    $$
    In particular, the map 
    $$(p,z)\mapsto \left(\mu_1(p,z,\theta),\ \frac{\partial \mu_1}{\partial z} (p,z,\theta), \mu_2(p,z,\theta),\ \frac{\partial \mu_2}{\partial z}(p,z,\theta)\right)$$ is continuous. 
\end{lemma}
\begin{proof}
	Fix $\theta\in(0,\infty)$. The continuity of $\mu_2$ follows easily from the Dominated Convergence Theorem. Since the integrand is integrable at the origin, we have $\frac{\partial\mu_2}{\partial z}(p,z)=\frac{z^2}{\sqrt{2f(z)}},$
	which is continuous in both variables. 
	
	As for $\mu_1$, we argue similarly to the proof of Lemma \ref{lemma:L'} (whose detailed computations, as recalled, are included in \cite{AgostinhoCorreiaTavares2}). We write
	\begin{align}\label{eqn.4.22_AUX}
		\sqrt{2}&\mu_1(p,z,\theta) \left((1-\theta^2)f(z)-f(1)\right)=-\int_z^{f_{2}^{-1}\left((1-\theta^2)f(z)\right)}t^2\sqrt{f(t)-(1-\theta^2)f(z)}dt\nonumber\\
		+&\int_z^{f_{2}^{-1}\left((1-\theta^2)f(z)\right)}\frac{(f(t)-f(1))t^2dt}{\sqrt{f(t)-(1-\theta^2)f(z)}},
	\end{align}
	and define $W:\R\setminus\{1\}\times\R\to\R$ by $W(x,y)=\frac{f(x)-f(1)}{f'(x)}x^2y$, which is of class $C^1$ in its domain and can be extended by continuity to $\R^2$. Consider now the curve $\gamma:\left[z,f^{-1}\left((1-\theta^2)f(z)\right)\right]\to \R^2$ defined by
	$\gamma(t)=(x(t),y(t))=\left(t,\sqrt{2}\sqrt{f(t)-(1-\theta^2)f(z)}\right),$
	which parametrizes the level curve $\Gamma$, $-y^2+f(x)=(1-\theta^2)f(z)$, for $y\geq0$.
	
	Then
	\begin{align*}
		-W(z,\theta\sqrt{2}\sqrt{f(z)})&=W\left(f^{-1}\left((1-\theta^2)f(z)\right),0\right)-W\left(z,\theta\sqrt{2}\sqrt{f(z)}\right)=\int_\Gamma 
		dW(x,y)\nonumber\\
		&=\int_\Gamma\frac{\partial}{\partial x}W(x,y)dx+\frac{\partial}{\partial y}W(x,y)dy\\
		&=\int_\Gamma y\left(2x\frac{f(x)-f(1)}{f'(x)}+x^2\left(1-\frac{f(x)-f(1)}{f'(x)^2}f''_p(x)\right)\right)dx+x^2\frac{f(x)-f(1)}{f'(x)}dy,
	\end{align*}
	and, taking $(x,y)=(x(t),y(t))=\gamma(t)$, we obtain
	\begin{align*}
		&\int_z^{f_{2}^{-1}\left((1-\theta^2)f(z)\right)}t^2\frac{f(t)-f(1)}{\sqrt{f(t)-(1-\theta^2)f(z)}}dt=-\sqrt{2}W(z,\theta\sqrt{2}\sqrt{f(z)})\\-&2\int_z^{f_{2}^{-1}\left((1-\theta^2)f(z)\right)}\sqrt{f(t)-(1-\theta^2)f(z)}\left(2t\frac{f(t)-f(1)}{f'(t)}+t^2\left(1-\frac{f(t)-f(1)}{f'(t)^2}f''_p(t)\right)\right)dt.
	\end{align*}
Equation \eqref{eqn.4.21} follows from plugging  the previous expression into \eqref{eqn.4.22_AUX}. 

For each fixed $\theta>0$, the continuity of the map $(p,z)\mapsto \mu_1(p,z,\theta)$ will follow by the same kind of arguments as in Lemma \ref{lemma:L'}. Let us compute its derivative.  Differentiating both sides of \eqref{eqn.4.21} with respect to $z$, we obtain
	\begin{align*}
		\sqrt{2}\frac{\partial\mu_1}{\partial z}&(p,z,\theta)((1-\theta^2)f(z)-f(1))=-\sqrt{2}(1-\theta^2)f'(z)\mu_1(p,z,\theta)+\frac{d}{dz}\left(-2\theta z^2\sqrt{f(z)}\frac{f(z)-f(1)}{f'(z)}\right)\nonumber\\
		-&\frac{d}{dz}\int_z^{f_{2}^{-1}((1-\theta^2)f(z))}g(t)\sqrt{f(t)-(1-\theta^2)f(z)}dt\nonumber\\
		&=-\sqrt{2}(1-\theta^2)f'(z)\mu_1(p,z,\theta)+\frac{d}{dz}\left(-2\theta z^2\sqrt{f(z)}\frac{f(z)-f(1)}{f'(z)}\right)+\theta g(z)\sqrt{f(z)}\nonumber\\
		+&(1-\theta^2)f'(z)\int_z^{f_{2}^{-1}((1-\theta^2)f(z))}\frac{g(t)}{2}\frac{dt}{\sqrt{f(t)-(1-\theta^2)f(z)}}.
	\end{align*}

	Computing the derivative of the non-integral term in the previous expression and simplifying, we get
	\begin{align}\label{mu1DerAux}
		\sqrt{2}\frac{\partial\mu_1}{\partial z}&(p,z,\theta)((1-\theta^2)f(z)-f(1))=
		-\frac{\theta z^2(f(z)-f(1))}{\sqrt{f(z)}}+2\theta z^2\frac{(f(z)-f(1))\sqrt{f(z)}f''(z)}{{f'}(z)^2}\nonumber\\
		-&4\theta z\frac{(f(z)-f(1))\sqrt{f(z)}}{f'(z)}
		-2\theta z^2\sqrt{f(z)}+3\theta z^2\sqrt{f(z)}-2\theta z^2\frac{(f(z)-f(1))\sqrt{f(z)}f''(z)}{{f'}(z)^2}\nonumber\\
		+&4\theta z\frac{(f(z)-f(1))\sqrt{f(z)}}{f'(z)}-(1-\theta^2)f'(z)\int_z^{f_{2}^{-1}((1-\theta^2)f(z))}\left(t^2-\frac{g(t)}{2}\right)\frac{dt}{\sqrt{f(t)-(1-\theta^2)f(z)}}\nonumber\\
		&=\theta z^2\frac{f(1)}{\sqrt{f(z)}}-(1-\theta^2)f'(z)\int_z^{f_{2}^{-1}((1-\theta^2)f(z))}\left(t^2-\frac{g(t)}{2}\right)\frac{dt}{\sqrt{f(t)-(1-\theta^2)f(z)}}.
	\end{align}

	Now, we prove the continuity of the map $(p,z)\mapsto\frac{\partial\mu_1}{\partial z}(p,z,\theta)$. This is trivial in the case $\theta=1$; therefore, we assume $\theta\neq 1$ in what follows. 
	
	We consider the integral term in \eqref{mu1DerAux}. Integrating this term by parts with respect to $t$, we get
	\begin{equation}\label{New_Int_EXPRESSION}
		-\frac{\theta(2z^2-g(z))\sqrt{f(z)}}{f'(z)}-\int_z^{f_{2}^{-1}((1-\theta^2)f(z))}\frac{d}{dt}\left(\frac{2t^2-g(t)}{f'(t)}\right)\sqrt{f(t)-(1-\theta^2)f(z)}.
	\end{equation}
	The non-integral term is continuous in both variables. Now, we observe that integrand function in this new integral expression contains no singularities, since
	$$\lim\limits_{t\to1}\frac{d}{dt}\left(\frac{2t^2-g(t)}{f'(t)}\right)  =\lim\limits_{t\to1}\frac{(4t-g'(t))f'(t)-(2t^2-g(t))f''(t)}{f'(t)^2}=\frac{22-11p+p^2}{6(p-2)}. $$ 
    Consequently, the continuity of the integral term in \eqref{New_Int_EXPRESSION} follows from the Dominated Convergence Theorem, since the integrand is uniformly bounded in compact sets that are far away from zero. This finishes the proof.
\end{proof}

\begin{lemma}[Explicit expression of $\frac{\partial \Theta_1}{\partial z}$ in the $\mathcal{T}$ and tadpole graphs.]\label{lemma:4.7.Temp}
	Let $p>2$ and $ 0<z<\left(\frac{p}{2}\right)^\frac{1}{p-2}$, and let $\Theta_1$ be as in \eqref{eq:Theta_1}. Then
	\begin{enumerate}
    \item For the $\mathcal{T}$-graph ($\theta=2$):
    \begin{align*}
		\sqrt{2}\frac{\partial\Theta_1}{\partial z}(p,z)&\left(-3f(z)-f(1)\right)=-6z^2\sqrt{f(z)}\\
		+&3f'(z)\int_z^{f_{2}^{-1}\left(-3f(z)\right)}\left(t^2-\frac{g(t)}{2}\right)\frac{dt}{\sqrt{f(t)+3f(z)}}.
	\end{align*}
	\item For the tadpole graph ($\theta=1/2$):
    \begin{align*}
		\sqrt{2}\frac{\partial\Theta_1}{\partial z}\left(p,z\right)&\left(\frac{3}{4}f(z)-f(1)\right)=\frac{3}{4}z^2\sqrt{f(z)}\\
		-&\frac{3}{4}f'(z)\int_z^{f_{2}^{-1}\left(3f(z)/4\right)}\left(t^2-\frac{g(t)}{2}\right)\frac{dt}{\sqrt{f(t)-\frac{3}{4}f(z)}}.
	\end{align*}
\end{enumerate}
    In particular, the map $(p,z)\mapsto \frac{\partial \Theta_1}{\partial z}(p,z)$ is continuous.
\end{lemma}
\begin{proof}
	Recalling the expression of $\Theta_1$ from Lemma \ref{lemma:Theta_1_expressions} in each graph, this is a direct consequence of Lemma \ref{lemma:mu_1reg}.
\end{proof}

\section{Grillakis-Shatah-Strauss theory in the $\mathcal{T}$ and tadpole graphs}\label{sec:GSS_particular}

In this section, we check that conditions $(A1)$, $(A2)$ and $(A3)$ from Section \ref{sec:GSS} are valid in the $\mathcal{T}$ graph $\mathcal{T}_1$ (terminal edge of length 1) and in the tadpole graph $\mathcal{G}_1$ (terminal loop of length 2). Recall that we can assume that $\ell=1$, up to a scaling.

\begin{lemma}\label{lemma:A1}
Let $\mathcal{H}_1$ be either the $\mathcal{T}$-graph $\mathcal{T}_1$ or the tadpole graph $\mathcal{G}_1$. Then, condition $(A_1)$ is satisfied.
\end{lemma}
\begin{proof}
The tadpole and $\mathcal{T}-$graphs are a part of a more general family of graphs, called starlike graphs. This family of graphs was introduced in \cite{cacciapuoti2017ground} and it consists of metric graphs with \textit{any} compact core and at least one half-line. In \cite[Proposition 2.2]{cacciapuoti2017ground} the authors prove, via fixed point arguments, that the IVP associated to the equation \eqref{NLSE_Chapter5} is locally well-posed in $H^1$ for any $p>2$. Moreover, due to energy and mass conservation, see \cite[Proposition 2.3]{cacciapuoti2017ground}, a global well-posedness result in $H^1$ is obtained for any $L^2$-subcritical power, see \cite[Theorem 2.3]{cacciapuoti2017ground}.
\end{proof}

Assumption $(A_2)$ concerns the existence of a regular curve of bound-states. Recall from Theorem \ref{thm:uniqueness} that in $\mathcal{H}_1$, for each $\lambda>0$, there exists exactly one positive action ground-state, which we denote by $u^\lambda$. 
 
\begin{lemma}\label{lemma:5.32}
	The curve $\lambda\in (0,\infty)\mapsto u^\lambda\in H^1(\mathcal{H}_1)$ is of class $C^1$. In particular, condition $(A_2)$ is satisfied.
\end{lemma}
\begin{proof}
	Recall from Theorem \ref{thm:description_edge_by_edge} the definition of $y(p,\ell)$ and $z(p,\ell)=\varphi(y(p,\ell))$. By Lemma \ref{cor:4.2.Temp}, these maps are $C^1$ in the variable $\ell$. Since $u_{h,\ell}$ is explicitly given as
    \[
    u_{h,\ell}(x)=\varphi(x+y(p,\ell)),
    \]
    the action ground-state depends on $\ell$ as a $C^1$ map. In the compact edge, $u_{\kappa,\ell}$ satisfies \eqref{eq:sol_finite_edge}, for which one has $C^1$ dependence on the parameters (see, for instance, \cite[Proposition 9.4]{amann1990ordinary}).  Finally, since $\ell=\sqrt{\lambda}$, the conclusion follows.
\end{proof}

The verification of $(A_3)$ is, as usual, much more delicate, and we dedicate the following section to it. 

\subsection{Spectral Analysis of $H_\lambda$: condition $(A_3)$}\label{AssumptionA_3}

The goal of this section is to prove that assumption $(A_3)$ holds true for both the tadpole and $\mathcal{T}$-graphs. Recalling that the length function $L$ is defined in \eqref{eq:Length}, we show the following.

\begin{proposition}\label{prop:5.26}
Consider the metric graph $\mathcal{H}_1$. For $p>2$ fixed, set\footnote{{Observe that $\lambda^*$  depends on the considered graph through the parameter $\theta$ (recall formula \eqref{eq:Length}).}} $\lambda^*:=(L(p,1))^2$. For any interval $J\subset\R^+\setminus\{\lambda^*\}$, the operator $H_\lambda:X\to X'$ satisfies assumption $(A_3)$.
\end{proposition}

Here, we benefit from the fact that the curve $\lambda\mapsto u^\lambda$ consists of action ground-states. This makes the negative eigenvalue property particularly simple to analyze, see Lemma \ref{prop:5.43.Temp}. However, analyzing the kernel of $H_\lambda$ is the biggest challenge as it depends on the topology of the graph we are considering. Nonetheless, inspired by the work done in \cite{noja2020standing}, which treats the case $p=6$ in the tadpole graph, we are able to determine the kernel of $H_\lambda$ (for all but one frequency $\lambda^*$).

The argument of the proof of Proposition \ref{prop:5.26} is almost exactly the same for both graphs, except for the analysis of the zero eigenvalue of $H_\lambda$, see the proof of Lemma \ref{prop:5.43.Temp} below. We start with the results that are common to both graphs.

\smallbreak

We begin by writing explicitly the operator $H_\lambda$. Take $X=H^1(\mathcal{H}_1;\C)\simeq H^1(\mathcal{H}_1)\times H^1(\mathcal{H}_1)$ with the usual (real) inner product and, for any, $u\in X$,  write $u\simeq\left(\Re(u),\Im(u)\right)$, where $\Re$ and $\Im$ denote the real and imaginary parts of $u$, respectively. For each $\lambda>0$, let $u^\lambda\simeq(u^\lambda,0)$ be the positive action ground-state in $\mathcal{H}_1$. Then, for any $v,w\in H^1(\mathcal{H}_1;\C)$, $H_\lambda$ is given by:
\begin{align}\label{LinearizedOperatorExplicit}
	\langle H_\lambda v,w\rangle=\langle S_\lambda''(u_\lambda,\mathcal{H}_1)v,w\rangle&=\langle \mathcal{L}^{-}_\lambda v,w\rangle+\langle \mathcal{L}^{+}_\lambda v,w\rangle,
\end{align}
where the operators $\mathcal{L}^{-}_\lambda, \mathcal{L}^{+}_\lambda:X\to X'$ are defined by:
\begin{equation}\label{L-Operator}
	\langle \mathcal{L}^{-}_\lambda v,w\rangle=\int_{\mathcal{H}_1}\Re(v')\Re(w')+(\lambda-(p-1)|{u^\lambda}|^{p-2})\Re(v)\Re(w)dx,\ \text{for all}\ w\in X,
\end{equation}
and
\begin{equation}\label{L+Operator}
	\langle \mathcal{L}^{+}_\lambda v,w\rangle=\int_{\mathcal{H}_1}\Im(v')\Im(w')+(\lambda-|{u^\lambda}|^{p-2})\Im(v)\Im(w)dx,\ \text{for all}\ w\in X.
\end{equation}
It is classical to show that $H_\lambda:X\to X'$ is self-adjoint (i.e., by taking the canonical identification $I:X\to X'$ via Riesz' representation theorem, $I^{-1}H_\lambda$ is a bounded self-adjoint operator in the standard sense).

Under the previous identification, we define the \textit{spectrum} of the operator $H_\lambda$ by
\begin{equation}\label{Spectrum}
	\sigma(H_\lambda):=\left\{\lambda\in\C: I^{-1}H_\lambda-\lambda\ \text{is not invertible}\right\}.
\end{equation}
Since $H_\lambda$ is self-adjoint, $\sigma(H_\lambda)\subset\R$, and we split the spectrum into the disjoint union $\sigma(H_\lambda)=\sigma_{disc}(H_\lambda)\cup\sigma_{ess}(H_\lambda)$, where
\begin{equation}\label{DiscreteSpec}
	\sigma_{disc}(H_\lambda)=\left\{\nu\in\R: \nu\ \text{is an eigenvalue of}\ H_\lambda\ \text{with finite multiplicity}\right\}
\end{equation}
is the \textit{discrete spectrum} and 
\begin{equation}\label{EssentialSpec}
	\sigma_{ess}(H_\lambda)=\sigma(H_\lambda)\setminus\sigma_{disc}(H_\lambda)
\end{equation}
is the \textit{essential spectrum}. We analyze the spectrum of $H_\lambda$ by examining the spectra of $\mathcal{L}^{-}_\lambda$ and $\mathcal{L}^{+}_\lambda$ separately. We start with the essential spectrum. 

\begin{lemma}\label{prop:5.40.Temp}
	For any $\lambda>0$,  $\sigma_{\text{ess}}\left(\mathcal{L}^{-}_\lambda\right)=\sigma_{\text{ess}}\left(\mathcal{L}^{+}_\lambda\right)=\sigma_{\text{ess}}\left(-\frac{d^2}{dx^2}+\lambda\right)=[\lambda,\infty).$ In particular, $\sigma_{\text{ess}}\left(H_\lambda\right)=[\lambda,\infty)$.
\end{lemma}
\begin{proof}
	We claim that $\mathcal{L}^{-}_\lambda$ and $\mathcal{L}^{+}_\lambda$ are compact perturbations of the operator $-\frac{d^2}{dx^2}+\lambda$, in which case the first two equalities will follow from Weyl's Theorem, see \cite[Section XIII.4-Example 3]{reed1978iv}. Finally, we note that the last equality holds as an adaptation of \cite[Theorem 2.3.5]{hofmann2021}. To prove the claim, it is enough to show that the operator $K:H^1(\mathcal{H}_1)\to H^1(\mathcal{H}_1)'$ defined by
	$$\langle K(u),v\rangle=\int_{\mathcal{H}_1}|u^\lambda|^{p-2}uvdx,\ \text{for all}\ v\in H^1(\mathcal{H}_1),$$
	is compact. We note that through the identification of $L^2(\mathcal{H}_1)$ with its dual and the exponential decay of the real line soliton $\varphi$, the operator $K$ can be understood, on each edge, as $K:H^1(\mathcal{H}_1)\to L^2(\mathcal{H}_1)$ and defined as $K(u)=|u^\lambda|^{p-2}u$.
		
	 On the compact core of the graph, by Rellich-Kondrachov embedding theorem,  the operator $K$ is compact. We show that it is compact on the half-lines. Take  $(u_n)_{n\in\N}$ to be a bounded sequence in $H^1(\R^+)$ such that $u_n \rightharpoonup u^*$ weakly in $H^1(\R^+)$, and let $M>0$ be such that $\|u_n\|_{H^1(\R^+)},\|u\|_{H^1(\R^+)}\leq C$. Denote $g:=|u^\lambda|^{p-2}\in W^{1,\infty}(\R^+)$, and recall that $u^\lambda$ is a decreasing portion of the real line soliton (see Theorem \ref{thm:description_edge_by_edge}), which is known to have exponential decay as $|x|\to \infty$. Then
	 \begin{align*}
	 	&\int_0^\infty\left|{g u_n-g u^*}\right|^2dx\leq g (0)\int_0^r\left|{u_n-u^*}\right|^2dx+\int_r^\infty\left|{g(u_n-u^*)}\right|^2dx\leq g (0)\int_0^r\left|{u_n-u^*}\right|^2dx+g (r)2M.
	 \end{align*}
By taking the limit as $n\to \infty$, and then as $r\to \infty$, we deduce that $gu_n\to gu^*$ in $L^2(\R^+)$ as $n\to \infty$.
\end{proof}

We now analyze the discrete spectrum of $\mathcal{L}^{-}_\lambda$ and $\mathcal{L}^{+}_\lambda$. Observe that, due to the exponential decay of the action ground-state on half-lines (Theorem \ref{thm:description_edge_by_edge}), the discrete spectrum of the operators $\mathcal{L}^{-}_\lambda$ and $\mathcal{L}^{+}_\lambda$ consists of a finite number of isolated eigenvalues. This was shown in \cite{kairzhan2018nonlinear}, in the context of star graphs, and referred often throughout the literature for other classes of graphs. For instance, \cite{noja2020standing}, for tadpole graphs, and, more generally, in \cite{cacciapuoti2017ground}, for starlike graphs. To analyze the discrete spectrum in more detail, we focus on understanding the eigenvalue problem $H_\lambda\chi=\nu\chi$, where  $\chi=(\chi_1,\chi_2)\in X$. This problem is expressed as follows:
\begin{equation}\label{eigenvalueProbs}
	\begin{cases}
		\int_{\mathcal{H}_1}\chi_1'\Re(w')+\lambda\chi_1\Re(w)-(p-1)|{u^\lambda}|^{p-2}\chi_1\Re(w)dx=\nu\int_\G\chi_1\Re(w)dx,\\
		\int_{\mathcal{H}_1}\chi_2'\Im(w')+\lambda\chi_2\Im(w)-|{u^\lambda}|^{p-2}\chi_2\Im(w)dx=\nu\int_\G\chi_2\Im(w)dx,
	\end{cases}\ \text{for any}\ w\in X.
\end{equation}

Let us first focus on (real) solutions $\chi_2\in H^1(\mathcal{H}_1)$ of the second equation in \eqref{eigenvalueProbs}. 
\begin{lemma}\label{lemma:5.41.Temp}
	For any $\lambda>0$, $\sigma_{disc}\left(\mathcal{L}^{+}_\lambda\right)\subseteq[0,\infty)$ and $\text{Ker}\left(\mathcal{L}^{+}_\lambda\right)=\text{span}\left\{u^\lambda\right\}$.
\end{lemma}
\begin{proof}
	We start by observing that $0$ is indeed an eigenvalue of $H_\lambda$ (and also $\mathcal{L}^{+}_\lambda$), since $w=(0,u^\lambda)\simeq iu^\lambda$ solves \eqref{eigenvalueProbs} with $\nu=0$.
	
	Next, we show that the first eigenvalue $\nu_1$ has to be zero. By the variational characterization of the first eigenvalue we have that
	\begin{equation}\label{MinimizationAux}
		\nu_1=\inf_{\chi_2\in H^1(\mathcal{H}_1)}\left\{\int_{\mathcal{H}_1}{\chi_2'}^2+\left(\lambda-{|u^\lambda|}^{p-2}\right)\chi_2^2dx: \int_{\mathcal{H}_1}\chi_2^2dx=1\right\}.
	\end{equation}
and any associated eigenfunction $\chi_2$ is therefore non-negative. Testing \eqref{eigenvalueProbs} with $w=(0,u^\lambda)$ and integrating by parts, we find that, since $u^\lambda>0$ is a bound-state,
	\begin{equation*}
		0=\int_{\mathcal{H}_1}\left(-({u^\lambda})''+\lambda u^\lambda-{(u^\lambda)}^{p-1}\right)\chi_2dx=\nu_1\int_{\mathcal{H}_1} u^\lambda\chi_2dx.
	\end{equation*}
which implies that $\nu_1=0$.
	
	To finish the proof, observe that $\nu_1=0$, being the first eigenvalue of $\mathcal{L}^{+}_\lambda$, is a simple eigenvalue, and thus $\text{Ker}(\mathcal{L}^{+}_\lambda)=\text{span}\left\{u^\lambda\right\}$.
\end{proof}

To finish the spectral analysis of $H_\lambda$, we need to look carefully into the spectrum of  $\mathcal{L}^{-}_\lambda$. We will need the following concepts.

\begin{definition}[\textit{Morse and Nullity Index}]
	Suppose that $A:X\to X'$ is a linear operator. We define \setlist{nolistsep}
	\begin{itemize}[noitemsep]
		\item the \textit{Morse Index} of $A$, as
		$ n(A):=\#\left\{\nu\in\R: \nu\ \text{is a negative eigenvalue of}\ A\right\}.$
		\item the \textit{Nullity Index} of $A$, as the multiplicity of the zero eigenvalue of $A$, and denote it by $z(A)$. 
	\end{itemize}
\end{definition}

With the above definition,
\begin{lemma}\label{lemma:5.44.Temp}
	For each $\lambda>0$, $n(H_\lambda)=n(\mathcal{L}^{-}_\lambda)=1$.
\end{lemma}
\begin{proof}
	Firstly, recall that $u_\lambda\simeq (u^\lambda,0)\in X$ is a (real) bound-state, so
	\begin{equation*}
		\langle H_\lambda u_\lambda,u_\lambda\rangle=\langle\mathcal{L}^{-}_\lambda u_\lambda,u_\lambda\rangle=\int_{\mathcal{H}_1}(-{u^\lambda}''+\lambda u^\lambda-(p-1){|u^\lambda|}^{p-1})u^\lambda dx=-(p-2)\|u^\lambda\|^p_{L^p(\mathcal{H}_1)}<0.
	\end{equation*}
	Consequently, $n(H_\lambda)\geq1$.
	
	To show the reverse inequality, observe that $u^\lambda$ is the (unique) solution of the minimization problem:
	\begin{equation*}
		\min\left\{I_\lambda(u,\mathcal{H}_1):=\frac{1}{2}\int_{\mathcal{H}_1}|u'|^2+\lambda |u|^2dx:\ R(u,\mathcal{H}_1):=\frac{1}{p}\int_{\mathcal{H}_1}|u|^pdx=\mu\right\}
	\end{equation*}
	for some suitable $\mu>0$ (for the $\mathcal{T}$-graph see \cite[Section 3.1]{AgostinhoCorreiaTavares}, for the tadpole see \cite[proof of Proposition 1.17]{AgostinhoCorreiaTavares}). Additionally, note that the action functional $S_\lambda$ can be written as $S_\lambda=I_\lambda-R$. Take the manifold of codimension one $R_\mu:=R^{-1}(\cdot,\mathcal{H}_1)\{\mu\}\subset X$, associated to the constraint $R(\cdot,\mathcal{H}_1)=\mu$.
	
Consider $v\in T_{u_\lambda}R_\mu$, the tangent space of $R_\mu$ at $u_\lambda$. Define a smooth curve $\gamma:[-1,1]\to R_\mu$ such that $\gamma(0)=u_\lambda$ and $\gamma'(0)=v$. Since $I_\lambda$ is of class $C^2$ in $X$ and $R$ is constant along $\gamma$, then $S_\lambda\circ \gamma=(I_\lambda-R)\circ \gamma$ is (twice) differentiable at $t=0$ and
	$$\frac{d^2}{dt^2}\left(S_\lambda\circ\gamma\right)\big|_{t=0}=\langle S_\lambda''(\Phi_\lambda)v,v\rangle=\langle I_\lambda''(\Phi_\lambda)v,v\rangle=\frac{d^2}{dt^2}\left(I_\lambda\circ\gamma\right)\big|_{t=0}=\langle I_\lambda''(\Phi_\lambda)v,v\rangle\geq0,$$
by minimality of $u^\lambda=(u^\lambda,0)$. Since $R_\mu$ has codimension $1$, we deduce that $n(H_\lambda)\leq 1$. Finally, from Lemma \ref{lemma:5.41.Temp}, we have that $n(H_\lambda)=n(\mathcal{L}^{-}_\lambda)$.
\end{proof}

\begin{lemma}\label{prop:5.43.Temp}
	Let $L$ be the length function defined in \eqref{eq:Length} (with  $\theta=2$ in the $\mathcal{T}$-graph $\mathcal{T}_\ell$, and $\theta=1/2$ in the tadpole graph $\mathcal{G}_\ell$). For each $\lambda>0$ with  $\lambda\neq\lambda^*:=(L_1(p,1))^2$, we have $z\left(\mathcal{L}^{-}_\lambda\right)=0$.
\end{lemma}

Assuming Lemma \ref{prop:5.43.Temp} is true, we show that assumption $(A_3)$ holds.
\begin{proof}[Proof of Proposition \ref{prop:5.26}]
	By Lemmas \ref{lemma:5.41.Temp}, \ref{lemma:5.44.Temp} and \ref{prop:5.43.Temp}, we have that $H_\lambda$ has exactly one negative eigenvalue, and its kernel is generated by $iu^\lambda$. Recalling that eigenvalues are in finite number, it follows from Lemma \ref{prop:5.40.Temp} that the remaining part of the spectrum is bounded away from zero.
\end{proof}

We now focus on proving Lemma \ref{prop:5.43.Temp}, following the arguments developed in \cite{noja2020standing} for the tadpole graph and $p=6$. We analyze the zero eigenvalue of $\mathcal{L}^{-}_\lambda$. 

\begin{remark}\label{rem:NK} Observe that, whenever $u\in H^1(\G)$ is a zero-eigenvector for $\mathcal{L}^-_\lambda$, then $u\in H^2_{\text{NK}}(\G)$, i.e., the restriction of $u$ to each edge, $u_e$, belongs to $H^2$ in each edge, and $u$ satisfies the Neumann-Kirchoff conditions at each vertex:
$$
\sum_{e \prec v} \frac{d u_e}{d x_e}(v)=0
$$
where $e \prec \mathbf{v}$ means that the edge $e$ is incident to the vertex $\mathbf{v}$, and the derivatives are taken in the inward direction of each edge.  
\end{remark}

Since the topology of the $\mathcal{T}$ and the tadpole graphs are different, we will need to split the proof of Lemma \ref{prop:5.43.Temp} in two parts, one for each graph.
Before that, we present an auxiliary result concerning the half-line.

\begin{lemma}\label{lemma:Half-lineKernel}
Take $y\geq 0$, and let $\varphi$ be the even real line soliton. Any non-trivial solution in $H^1$ to the equation
	\begin{equation}\label{eqn.LinearizedHalf-line}
		-v''+v-(p-1)\varphi(x+y)^{p-2}v=0,\ \text{in}\ \R^+,
	\end{equation}
	is of the form $v(x)=\alpha\varphi'(x+y)$, for some  $\alpha\in\R$. 
\end{lemma}
\begin{proof}
	We follow \cite[Proposition 2.8b)]{weinstein1985modulational}. Without loss of generality, we assume $y=0$, being the proof in the general case completely analogous. 
	
	First of all, differentiating the equation $-\varphi''+\varphi=\varphi^p$, we get that $\varphi'$ is a solution of \eqref{eqn.LinearizedHalf-line}. Let $v$ be a solution to the ODE \eqref{eqn.LinearizedHalf-line}, with $v(0)=1$ and $v'(0)=0$;. 
    In particular, any solution to \eqref{eqn.LinearizedHalf-line} is of the form
	\[
	\alpha\varphi'+\beta v,\qquad \text{ for some } \alpha,\beta\in \R.
	\]
	We show that $v\notin H^1$, which concludes our proof. Consider the Wronskian of the two solutions, 
	$$ 
W(\varphi',v)(x)=\text{det}\begin{pmatrix}
		\varphi'(x) & v(x) \\ \varphi''(x) & v'(x)
	\end{pmatrix}=\varphi'(x)v'(x)-v(x)\varphi''(x). $$
	Since both $v$ and $\varphi'$ are solutions of \eqref{eqn.LinearizedHalf-line}, we get
	$$ \frac{d}{dx} W(\varphi',v)(x)=\frac{d}{dx}\left(\varphi'(x)v'(x)-v(x)\varphi''(x)\right)=\varphi'(x)v''(x)-v(x)\varphi'''(x)=0,$$
so $W(\varphi',v)(x)\equiv-\varphi''(0)$.
	
	Observe now that, for $x>0$, we have
	$$ \left(\frac{v}{\varphi'}\right)'(x)=\frac{W(\varphi',v)(x)}{{\varphi'}^2(x)}=\frac{-\varphi''(0)}{{\varphi'}^2(x)}.$$
Fixing $x_0$ such that $\varphi(x_0)=1$, we have $\varphi''>0$ for $x\geq x_0$. Then
	\begin{equation*}
		\frac{v(x)}{\varphi'(x)}-\frac{v(x_0)}{\varphi'(x_0)}=-\varphi''(0)\int_{x_0}^x\frac{dt}{{\varphi'}^2(t)} \quad \mbox{that is, }\quad v(x)=v(x_0)\frac{\varphi'(x)}{\varphi'(x_0)}-\varphi''(0)\varphi'(x)\int_{x_0}^x\frac{dt}{{\varphi'}^2(t)},
	\end{equation*}
As $x\to\infty$, since $\varphi(x)=\left(\frac{p}{2}\right)^\frac{1}{p-2}\sech^{\frac{2}{p-2}}\left(\frac{p-2}{2}x\right),$ we find that $v(x)\to \infty$ and thus $v\not\in H^1$.
\end{proof}

Next, we study the Nullity Index in $\mathcal{T}_1$.

\begin{proof}[Proof of Lemma \ref{prop:5.43.Temp} for the $\mathcal{T}$-graph $\mathcal{T}_1$.]
	We consider a solution $\Gamma\in H^1(\mathcal{T}_1)$ to the eigenvalue problem $\mathcal{L}^{-}_\lambda\Gamma=0$, and show that $\Gamma=0$.  Denote by $\Gamma_1,\Gamma_2$ the restriction of $\Gamma$ to each half-line, and $\Gamma_3$ the restriction of $\Gamma$ to the pendant. Observe that, by Remark \ref{rem:NK}, $\Gamma\in H^2_{\text{NK}}(\mathcal{T}_1)$.	

    \medskip
	     \noindent\textit{Step 1. Discussion of the kernel of $\mathcal{L}^{-}_\lambda$ in terms of the parameter $\ell$.}
	
	Recall the scaling
	$u^\lambda(x)=\lambda^{\frac{1}{p-2}}u_\ell(\lambda^{1/2}x)$ and define, for $\ell=\lambda^\frac{1}{2}$, $\gamma=(\gamma_1,\gamma_2,\gamma_3)\in H^2_{\text{NK}}(\mathcal{T}_\ell)$ such that
	\begin{equation*}
		\begin{cases}
			\Gamma_i(x)=\gamma_i(\lambda^\frac{1}{2} x),\ x\in[0,\infty),\ i=1,2,\\
			\Gamma_3(x)=\gamma_3(\lambda^\frac{1}{2} x),\ x\in[0,\ell].
		\end{cases}
	\end{equation*}
	Then, the  problem $\mathcal{L}^{-}_\lambda\Gamma=0$ can be written equivalently in terms of $\gamma$ as
	\begin{equation}\label{eqn.4.eigenvalueprobTGRAPH}
		\begin{cases}
			-\gamma_i''+\gamma_i-(p-1)u_{\ell,i}^{p-2}\gamma_i=0,\ \text{in}\ [0,\infty),\ i=1,2;\\
			-\gamma_3''+\gamma_3-(p-1)u_{\ell,3}^{p-2}\gamma_3=0,\ \text{in}\ [0,\ell],
		\end{cases}\text{with}\ \ \begin{cases}
			\gamma_1(0)=\gamma_2(0)=\gamma_3(0),\\
			\gamma_1'(0)+\gamma_2'(0)+\gamma_3'(0)=0,\\
			\gamma_3'(\ell)=0.
		\end{cases}
	\end{equation}

    \medskip
	\noindent\textit{Step 2. Analysis of the boundary value problem \eqref{eqn.4.eigenvalueprobTGRAPH}.}

	We start by analyzing the half-lines. Recall from Theorem \ref{thm:description_edge_by_edge} that, for each $i=1,2$, we have $u_{\ell,i}(x)=\varphi(x+y(p,\ell))$, for some $y(p,\ell)>0$ uniquely defined in terms of $\ell$ by the equation
	$\varphi(y(p,\ell))=L^{-1}(p,\ell)$, where $L$ is the length function with $\theta=2$, and $\varphi$ is the real line soliton.

	Let $i=1,2$. By Lemma \ref{lemma:Half-lineKernel}, the function  $\gamma_i(x)=\alpha_i\varphi'(x+y(p,\ell))$ is the only non-trivial $H^1$ solution to the first equation in \eqref{eqn.4.eigenvalueprobTGRAPH} for some $\alpha_i\in\R$. The continuity condition $\gamma_1(0)=\gamma_2(0)$ then yields, for some $\alpha \in \mathbb{R}$,
	\begin{equation}\label{eqn.gamma1TGRAPH}
		\gamma_i(x)=\alpha\varphi'(x+y(p,\ell)),\ \text{for}\  i=1,2.
	\end{equation}
	
	We focus now on the terminal edge. Similarly, since $u_\ell$ is a bound-state, $\gamma_3=u'_{\ell,3}$ solves the second equation in \eqref{eqn.4.eigenvalueprobTGRAPH}. Since we have a linear second-order ODE, any solution to the second equation in \eqref{eqn.4.eigenvalueprobTGRAPH} takes the form
	\begin{equation}\label{eqn.gamma2TGRAPH}
		\gamma_3(x)=\beta u'_{\ell,3}(x)+\delta W(x),
	\end{equation}
	where $\beta,\delta\in \mathbb{R}$ are arbitrary and $W$ is a linear independent solution from $u'_{\ell,3}$ to be chosen later.
	
	Since $u_\ell$ satisfies the Neumann-Kirchoff conditions, by equations \eqref{eqn.gamma1TGRAPH} and \eqref{eqn.gamma2TGRAPH}, the boundary conditions in \eqref{eqn.4.eigenvalueprobTGRAPH} become
	\begin{equation*}
		\begin{cases}
			\beta u''_{3,\ell}(\ell)+\delta W'(\ell)=0,\\
			\alpha\varphi'(y(p,\ell))=-2\beta\varphi'(y(p,\ell))+\delta W(0),\\
			2\alpha\varphi''(y(p,\ell))+\beta u''_{\ell,3}(0)+\delta W'(0)=0.
		\end{cases}
	\end{equation*}
	Recalling that, for $f:\R^+\to\R$ given by $f(x)=\frac{1}{2}x^2-\frac{1}{p}x^p$, we have $u''_{\ell,3}=f'(u_{\ell,3})$ and $\varphi''=f'(\varphi)$. The previous system can be reduced to
	\begin{equation}\label{NewBoundaryConds}
		\begin{cases}
			\beta f'(u_{\ell,3}(\ell))+\delta W'(\ell)=0,\\
			(\alpha+2\beta)\varphi'(y(p,\ell))=\delta W(0),\\
			(2\alpha+\beta)f'(\varphi(y(p,\ell)))+\delta W'(0)=0.
		\end{cases}
	\end{equation}
	
    \medskip
	\noindent\textit{Step 3. Study of $C(\ell)$.}
	
	Recall that, by \eqref{eq:hamilt} with $\theta=2$, $u_{\ell,3}$ satisfies, for every $x\in[0,\ell]$,
	\begin{equation}\label{eqn.ODEDependCTGRAPH}
		-\frac{1}{2}{u_{\ell,3}}'(x)^2+f({u_{\ell,3}}(x))=C(\ell):=-3f\left(\varphi(y(p,\ell))\right).
	\end{equation}
By Lemma \ref{cor:4.2.Temp}, the map $C:\R^+\to (-3f(1),0)$ is of class $C^1$, and $C'(\ell)=0$ if and only if $f'(\varphi(y(p,\ell)))=0$, i.e., if $\varphi(y(p,\ell))=1$. Taking then the unique $\ell^*>0$ such that $\varphi(y(p,\ell^*))=L^{-1}(p,\ell^*)=1$, the function $C$ is invertible in the intervals $(0,\ell^*)$ and $(\ell^*,\infty)$. 

Observe that, by Lemma \ref{cor:4.2.Temp}, $\frac{\partial y}{\partial\ell}(p,\ell)>0$. Since $\varphi$ is decreasing, we have, for every $\ell<\ell^*$, $\varphi(y(p,\ell))>1$. Conversely, if $\ell>\ell^*$, $\varphi(y(p,\ell))<1$. Thus, the inverses of $C$ are the functions $\ell_1:(-3f(1),0)\to(\ell^*,\infty)$ and $\ell_2:(-3f(1),0)\to(0,\ell^*)$ defined by
	$$ \ell_1(C)=y^{-1}\left(p,\varphi^{-1}\left(f_{1}^{-1}\left(-\frac{1}{3}C\right)\right)\right)\ \text{and}\ \ell_2(C)=y^{-1}\left(p,\varphi^{-1}\left(f_{2}^{-1}\left(-\frac{1}{3}C\right)\right)\right).$$
Since, for $i=1,2$, $$\lim_{C\to-3f(1)^+}\ell_i(C)=y^{-1}(p,x_0)=\ell^*,$$ the functions $\ell_i$ can be both continuously extended to the interval $[-3f(1),0)$.\footnote{However, since $C'(\ell^*)=0$, 
	$$\lim_{C\to-3f(1)^+}\left|\frac{\partial \ell_i}{\partial C}(C)\right|=+\infty,\quad i=1,2.$$
	Consequently, the derivative cannot be continuously extended to the interval $[-3f(1),0)$.}

\noindent	\textit{Step 4. Choice of $W$.}

	Fix now $i=1$ or $i=2$  and consider $\ell=\ell_i(C)$. To ease notation, we shall drop the subscript $i$.
Define $w(\cdot;C)$ as the unique (positive and monotone) solution to
	\begin{equation}\label{eqn.boundary_conditionsTGRAPH}
		\begin{cases}
			-{w''}+{w}={w}^{p-1},\ \text{in}\ [0,\ell],\\
			{w}(0;C)=\varphi(y(p,\ell)),\\
			{w}'(0;C)=-2\varphi'(y(p,\ell))=2 \sqrt{2 f(z(p,\ell))},\\
			{w}'(\ell,C)=0.
		\end{cases}
	\end{equation}
	By the previous step, this IVP has $C^1$ dependence on initial conditions for $C\neq-3f(1)=:C^*$ and continuous dependence on initial conditions in the case $C=C^*$. Observe also that $w(x;C)$  is related with $u_{\ell,3}$ by
	$$w(\cdot;C)=u_{\ell_i(C),3},\ i=1,2,$$
	where $i=1$ if $\ell>\ell^*$ and $i=2$ if $\ell<\ell^
	*$.
	
	For any $C\neq C^*$ (equivalently, $\ell\neq\ell^*$), define, for $x\in[0,\ell]$, $$W(x):=\frac{\partial w}{\partial C}(x;C).$$
By differentiation with respect to $C$ of the equation in \eqref{eqn.boundary_conditionsTGRAPH}, we obtain that $W$ is a solution of the second ODE in \eqref{eqn.4.eigenvalueprobTGRAPH}. We also observe that $w$ solves 
	\begin{equation}\label{ODEinC}
		-\frac{1}{2}w'(x;C)^2+f(w(x;C))=C,\ \text{for}\ x\in[0,\ell].
	\end{equation}
	Consequently, differentiation with respect to $C$ yields the following relation between $w$ and $W$:
	\begin{equation}\label{eqn.ConservAux}
		1=-w'(x;C)W'(x)+f'(w(x;C))W(x).
	\end{equation}
	
	We now check that $W$ is linearly independent from $u'_{\ell_i(C),3}$. Assume, by contradiction that there exists $k\neq0$ such that $W(x)=ku'_{\ell_i(C),3}(x)$ for any $x\in[0,\ell]$. 
	If $x=\ell$, we know that $w'(\ell;C)=u'_{\ell_i(C),3}(\ell)=0$, and thus, $W(\ell)=0$. Now, plugging $x=\ell$ into \eqref{eqn.ConservAux} gives us a contradiction.

    \medskip
\noindent	\textit{Step 5. $W(0)\neq0$.}

Differentiating with respect to $C$ the second equality in \eqref{eqn.boundary_conditionsTGRAPH}, we obtain, by Lemma \ref{cor:4.2.Temp} and the fact that $\ell\neq \ell^*$,
	\begin{equation}\label{eqn.W(0)AUXTGRAPH}
		W(0)=\varphi'(y(p,\ell(C)))\frac{\partial}{\partial\ell}y(p,\ell(C))\ell'(C)\neq0.
	\end{equation}
	
    \medskip
	\noindent\textit{Step 6.} $W'(0)\neq0$.

	Reasoning as in the previous step, we have, differentiating with respect to $C$ the third equality in \eqref{eqn.boundary_conditionsTGRAPH}, and since $\ell\neq \ell^*$,
	\begin{equation}\label{eqnStep4Aux2}
		W'(0)=-2f'(\varphi(y(p,\ell(C))))\frac{\partial}{\partial\ell}y(p,\ell(C))\ell'(C)\neq0.
	\end{equation}
	
    \medskip
	\noindent\textit{Step 7.} $W'(\ell)\neq0$. 
    
    Differentiating with respect to $C$ the last equation in \eqref{eqn.boundary_conditionsTGRAPH} yields
	$$0=w''(\ell(C);C)\ell'(C)+W'(\ell)=f'(w(\ell(C);C))\ell'(C)+W'(\ell).$$
	If $W'(\ell)=0$, we would obtain $f'(w(\ell(C);C))=0$, so (since $w$ is positive) $w(\ell(C);C)=1$. Moreover, given that $w'(\ell(C);C)=0$, by existence and uniqueness theory for ODE's, we have $w(x;C)=u_{\ell,3}(x)\equiv1$ in $[0,\ell]$, which is a contradiction, since the action ground-state $u_\ell$ is not constant on the terminal edge.

    \medskip
	\noindent\textit{Step 8. Conclusion.}
	
	On the one hand, from the second and third equations in \eqref{NewBoundaryConds}, we obtain
	$$\frac{W(0)}{W'(0)}=\frac{(\alpha+2\beta)\varphi'(y(p,\ell))}{-(2\alpha+\beta)f'(\varphi(y(p,\ell)))}.$$
	On the other hand, from equations \eqref{eqn.W(0)AUXTGRAPH} and \eqref{eqnStep4Aux2}, given that $\ell'(C)\neq0$ and $\frac{\partial y}{\partial\ell}(p,\ell)\neq0$, we have
	$$\frac{W(0)}{W'(0)}=\frac{\varphi'(y(p,\ell))}{-2f'(\varphi(y(p,\ell)))}.$$
	Equating both expressions for $\frac{W(0)}{W'(0)}$, we get
	$$\frac{\varphi'(y(p,\ell))}{-2f'(\varphi(y(p,\ell)))}=\frac{(\alpha+2\beta)\varphi'(y(p,\ell))}{-(2\alpha+\beta)f'(\varphi(y(p,\ell)))},$$
	which simplifies to $$\alpha+2\beta=\alpha+\frac{1}{2}\beta\Longleftrightarrow \beta=0.$$
	Thus, the system \eqref{NewBoundaryConds} becomes
	\begin{equation*}
		\begin{cases}
			\delta W'(\ell)=0,\\
			\alpha\varphi(y(p,\ell))=\delta W(0),\\
			2\alpha f'(\varphi(y(p,\ell)))=-\delta W'(0).
		\end{cases}\Longleftrightarrow
		\begin{cases}
			\delta =0,\\
			\alpha\varphi(y(p,\ell))=0,\\
			2\alpha f'(\varphi(y(p,\ell)))=0.
		\end{cases}\Longleftrightarrow\begin{cases}\delta=0,\\ \alpha=0.			\end{cases}
	\end{equation*}
	In conclusion, $\gamma=0=\Gamma$, and no nontrivial solution to the equation $\mathcal{L}^{-}_\lambda\Gamma=0$ exists in $H^2_{\text{NK}}(\mathcal{T}_1)$.
\end{proof}

Finally, we study the Nullity Index in $\mathcal{G}_1$. We follow the same structure as before, highlighting mostly the differences.

\begin{proof}[Proof of Lemma \ref{prop:5.43.Temp} for the tadpole graph $\mathcal{G}_1$.]
	Let $\G_1$ be the tadpole graph. We consider a solution $\Gamma\in H^1$ to the eigenvalue problem $\mathcal{L}^{-}_\lambda\Gamma=0$ and show that $\Gamma=0$. We denote by $\Gamma_1$ and $\Gamma_2$, respectively, the restrictions of $\Gamma$ to the half-line and to the loop. Observe that, by Remark \ref{rem:NK}, $\Gamma\in H^2_{\text{NK}}(\G_1)$.	
	
    \medskip
	\noindent\textit{Step 1. Discussion of the kernel of $\mathcal{L}^{-}_\lambda$ in terms of the parameter $\ell$.}
	
	Recall the scaling
	$u^\lambda(x)=\lambda^{\frac{1}{p-2}}u_\ell(\lambda^{1/2}x)$ and define, for $\ell=\ell(\lambda)=\lambda^\frac{1}{2}$, $\gamma=(\gamma_1,\gamma_2)\in H^2_{\text{NK}}(\G_\ell)$
	\begin{equation*}
		\begin{cases}
			\Gamma_1(x)=\gamma_1(\lambda^\frac{1}{2} x),\ x\in[0,\infty),\\
			\Gamma_2(x)=\gamma_1(\lambda^\frac{1}{2} x),\ x\in[0,2\ell].
		\end{cases}
	\end{equation*}
	Then, problem $\mathcal{L}\Gamma=0$ can be written equivalently in $(\gamma_1,\gamma_2)$ as
	\begin{equation}\label{eqn.4.eigenvalueprob}
		\begin{cases}
			-\gamma_1''+\gamma_1-(p-1)u_{\ell,1}^{p-2}\gamma_1=0,\ \text{in}\ [0,\infty),\\
			-\gamma_2''+\gamma_2-(p-1)u_{\ell,2}^{p-2}\gamma_2=0,\ \text{in}\ [0,2\ell],
		\end{cases}\text{with}\ \begin{cases}
			\gamma_1(0)=\gamma_2(0)=\gamma_2(2\ell),\\
			\gamma_1'(0)+\gamma_2'(0)-\gamma_2'(2\ell)=0.
		\end{cases}
	\end{equation}
	
    \medskip
\noindent	\textit{Step 2. Analysis of the boundary value problem \eqref{eqn.4.eigenvalueprob}.}
		
We have $u_{\ell,1}(x)=\varphi(x+y(p,\ell))$, with $y(p,\ell)>0$ uniquely defined in terms of $\ell$ by the equation
	$\varphi(y(p,\ell))=L^{-1}(p,\ell)$, where $L$ is the length function, this time with $\theta=\frac{1}{2}$. By Lemma \ref{lemma:Half-lineKernel}, for some $\alpha\in\R$,
	\begin{equation}\label{eqn.gamma1}
		\gamma_1(x)=\alpha\varphi'(x+y(p,\ell)).
	\end{equation}
Any general solution to the second equation in \eqref{eqn.4.eigenvalueprob} takes the form
	\begin{equation}\label{eqn.gamma2}
		\gamma_2(x)=\beta u'_{\ell,2}(x)+\delta W(x),
	\end{equation}
	where $\beta,\delta\in \R$ are arbitrary and $W$ is a linear independent solution from $u'_{\ell,2}$ to be chosen later.
	
    \medskip
	\noindent\textit{Step 3. Analysis of $C(\ell)$.}
	
	This is a simple adaptation of Step 3 of the proof in the $\mathcal{T}$-graph. Note that in $\G_\ell$, by \eqref{eq:hamilt} with $\theta=\frac{1}{2}$, $u_{\ell,2}$ satisfies, for every $x\in[0,2\ell]$,
	\begin{equation}\label{eqn.ODEDependC}
		-\frac{1}{2}{u_{\ell,2}}'(x)^2+f({u_{\ell,2}}(x))=C(\ell)=\frac{3}{4}f\left(\varphi(y(p,\ell))\right).
	\end{equation}
	Moreover, there exists a unique $\ell^*>0$ such that $C'(\ell^*)=0$, characterized by $\varphi(y(p,\ell^*))=L^{-1}(\ell^*)=1$. The function $C:\R^+\to\left(0,\frac{3}{4}f(1)\right)$ can be inverted for $\ell\neq\ell^*$. The inverses of $C$ are the functions $\ell_1:(0,\frac{3}{4}f(1))\to(\ell^*,\infty)$ and $\ell_2:(0,\frac{3}{4}f(1))\to(0,\ell^*)$ defined by
	$$ \ell_1(C)=y^{-1}\left(p,\varphi^{-1}\left(f_{1}^{-1}\left(\frac{4}{3}C\right)\right)\right)\ \text{and}\ \ell_2(C)=y^{-1}\left(p,\varphi^{-1}\left(f_{2}^{-1}\left(\frac{4}{3}C\right)\right)\right).$$

    \medskip
	\noindent\textit{Step 4. Choice of $W$.}
	
	Fix now $i=1$ or $i=2$ and consider $\ell=\ell_i(C)$ as given in the previous step. To ease notation we shall drop the subscript $i$. Let $w(\ \cdot\ ;C)$ be the (unique) positive and symmetric (with respect to $x=\ell$) solution to
	\begin{equation}\label{eqn.boundary_conditions}
		\begin{cases}
			-w''+w=w^{p-1},\ \text{in}\ [0,2\ell],\\
			w(0;C)=w(2\ell;C)=\varphi(y(p,\ell)),\\
			w'(0;C)=-\frac{1}{2}\varphi'(y(p,\ell))=\frac{1}{2}\sqrt{2f(z(p,\ell))},\\
		\end{cases}
	\end{equation}
	By the previous step, this IVP has $C^1$ dependence on initial conditions for $C\neq\frac{3}{4}f(1)=:C^*$ and continuous dependence on initial conditions in the case $C=C^*$. Observe also that the function $w(x;C)$ is related with $u_{\ell,2}$ by
	$$w(\cdot;C)=u_{\ell_i(C),2},\ i=1,2,$$
	where $i=1$ if $\ell>\ell^*$ and $i=2$ if $\ell<\ell^
	*$.
We define, for $x\in[0,\ell]$, $$W(x):=\frac{\partial w}{\partial C}(x;C).$$ Differentiating with respect to $C$, we obtain that $W$ is a solution to the second ODE in \eqref{eqn.4.eigenvalueprob} that is linearly independent from $u'_{\ell,2}$.

    \medskip
\noindent	\textit{Step 5. Reduction to the unique solvability of a linear $2\times 2$ system.}
	
	Observe that, by symmetry of $u_{\ell,2}$ (recall Theorem \ref{thm:uniqueness}),  $u_{\ell,2}(x)=u_{\ell,2}(2\ell-x)$, for any $x\in [0,2\ell]$. Thus, in the same interval, we have $u'_{\ell,2}(x)=-u'_{\ell,2}(2\ell-x)$.

	From the continuity condition in \eqref{eqn.4.eigenvalueprob}, equations \eqref{eqn.gamma1} and \eqref{eqn.gamma2}, and the symmetry properties of $u'_{\ell,2}$ and $W$, 
	\begin{equation}\label{Cont-EigenProb}
		\begin{cases}
			\gamma_2(2\ell)=\gamma_1(0),\\
			\gamma_2(0)=\gamma_1(0),\\
			\gamma_1'(0)+\gamma_2'(0)-\gamma_2'(2\ell)=0,
		\end{cases}
		\Longleftrightarrow\ \ 
		\begin{cases}
			\beta u'_{\ell,2}(2\ell)+\delta W(2\ell)=\alpha\varphi'(y(p,\ell)),\\
			-\beta u'_{\ell,2}(0)+\delta W(0)=\alpha\varphi'(y(p,\ell)),\\
			\alpha\varphi''(y(p,\ell))+\delta (W'(0)-W'(2\ell))=0.
		\end{cases}
	\end{equation}	
	Since  $u'_{\ell,2}(0)=-u'_{\ell,2}(2\ell)$ we have, by adding the first two equations in \eqref{Cont-EigenProb},
	\begin{equation}\label{seconlineareq}
		2\alpha\varphi'(y(p,\ell))=\delta\left(W(0)+W(2\ell)\right).
	\end{equation}
	We claim that the linear system 
	\begin{equation}\label{ClaimSystem}
		\begin{cases}
			\alpha\varphi''(y(p,\ell))+\delta (W'(0)-W'(2\ell))=0,\\
			2\alpha\varphi'(y(p,\ell))-\delta\left(W(0)+W(2\ell)\right)=0.
		\end{cases}
	\end{equation}
	has only the trivial solution $(\alpha,\delta)=(0,0)$. Observe that, if this is the case, then by the first equation in \eqref{Cont-EigenProb}, we find $\beta=0$ and therefore no nontrivial solution to the equation $\mathcal{L}^{-}_\lambda\Gamma=0$ exists in  $H^2_{\text{NK}}(\G_1)$.
	
    \medskip
	\noindent\textit{Step 6. Unique solvability of \eqref{ClaimSystem}.}
	
	The idea is to write the system \eqref{ClaimSystem} in terms of the soliton $\varphi$. Differentiating the second and third equations in \eqref{eqn.boundary_conditions} with respect to $C$, 
	\begin{equation}\label{W(0)-W(2ell)_relation}
		W(0)=2w'(2\ell;C)\ell'(C)+W(2\ell)=\varphi'(y(p,\ell))\frac{\partial y}{\partial\ell}(p,\ell(C))\ell'(C)
	\end{equation}
	and
	\begin{equation}\label{W'(0)_Explicit}
		W'(0)=-\frac{1}{2}\varphi''(y(p,\ell))\frac{\partial y}{\partial\ell}(p,\ell(C))\ell'(C)
	\end{equation}
	From \eqref{W(0)-W(2ell)_relation},
	\begin{equation}\label{W(2ell)_eq}
		W(2\ell)=\varphi'(y(p,\ell))\frac{\partial y}{\partial\ell}(p,\ell(C))\ell'(C)-\varphi'(y(p,\ell))\ell'(C).
	\end{equation}
	
	Recalling that $w(\cdot;C)=u_{\ell(C),2}$, we have, by the symmetry properties of $u_{\ell,2}$, that $w'(2\ell;C)=-w'(0;C)$. Differentiating this expression with respect to $C$ and using $u_{\ell,2}(0)=u_{\ell,2}(2\ell)=\varphi(y(p,\ell))$, we get
	\begin{align}\label{W'(0)-W'(2ell)_relation}
		-W'(0)&=2w''(2\ell;C)\ell'(C)+W'(2\ell)=-2f'(w(2\ell;C))\ell'(C)+W'(2\ell)\nonumber\\
		&=-2f'(u_{\ell,2}(0))\ell'(C)+W'(2\ell)=2\varphi''(y(p,\ell))\ell'(C)+W'(2\ell).
	\end{align}
	Now, \eqref{W'(0)_Explicit} and \eqref{W'(0)-W'(2ell)_relation} yield
	\begin{equation}\label{W'(2ell)_eq}
		W'(2\ell)=\frac{1}{2}\varphi''(y(p,\ell))\frac{\partial y}{\partial\ell}(p,\ell(C))\ell'(C)-2\varphi''(y(p,\ell))\ell'(C).
	\end{equation}

	Therefore, the proof of unique solvability of \eqref{ClaimSystem} reduces to checking that the matrix
	\begin{align*}
		A&=\begin{pmatrix}
			\varphi''(y(p,\ell)) & W'(0)-W'(2\ell) \\
			2\varphi'(y(p,\ell)) & -(W(0)+W(2\ell))
		\end{pmatrix}\\
		&=\begin{pmatrix}
			\varphi''(y(p,\ell)) & -\varphi''(y(p,\ell))\frac{\partial y}{\partial\ell}(p,\ell(C))\ell'(C)+2\varphi''(y(p,\ell))\ell'(C) \\
			2\varphi'(y(p,\ell)) & -2\varphi'(y(p,\ell))\frac{\partial y}{\partial\ell}(p,\ell(C))\ell'(C)+\varphi'(y(p,\ell))\ell'(C)
		\end{pmatrix},
	\end{align*}
 has nonzero determinant. By direct computation,
	 \begin{align*}
	 	\text{det}
	 	(A)=-3\varphi'(y(p,\ell))\varphi''(y(p,\ell))\ell'(C).
	 	\end{align*}
	 By construction of $\ell$ as a function of $C$, independently of taking $i=1$ or $i=2$, we know that $\varphi''(y(p,\ell(C))\neq0$. Since also $y(p,\ell)>0$, we have $\varphi'(y(p,\ell))\neq0$. Moreover, one can easily check by straightforward differentiation that, as a consequence of Lemma \ref{cor:4.2.Temp}, $\ell'(C)\neq0$ for any $C\neq C^*=\frac{3}{4}f(1)$ independently of choosing $i=1$ or $i=2$. In particular, it follows that $\text{det}(A)\neq0$ and the proof is finished.
\end{proof}	 

\section{Stability transitions near $p=6$}\label{sec:NonMonotoneTheta}

In view of having a stability transition, it is crucial to prove that the map $\lambda\mapsto\Theta(6,\lambda)$ is not monotone. For the tadpole graph, this was shown in \cite{noja2020standing}.

\begin{theorem}[{{\cite[Theorem 3]{noja2020standing}}}]\label{th:5.28}
	Let $\lambda>0$ and $u^\lambda$ be the action ground-state on the tadpole graph $\mathcal{G}_1$. Then, there exist $\lambda_1>0$ for which
\[
		\frac{\partial\Theta}{\partial\lambda}(6,\lambda)>0\quad\text{for}\quad \lambda\in(0,\lambda_1) \quad \text{and}\quad \frac{\partial\Theta}{\partial\lambda}(6,\lambda)<0\quad\text{for $\quad\lambda\in(\lambda_1,\infty),$}.
\]
\end{theorem}

Inspired by the arguments in \cite{noja2020standing}, we present the first complete characterization of the $L^2$-norm of action ground-states in the critical case $p=6$ on the $\mathcal{T}-$graph. This complements what was known in \cite{AgostinhoCorreiaTavares,pierotti2021local}. Namely, we prove

\begin{theorem}\label{th:5.29}
	Let $\lambda>0$ and $u^\lambda$ be the action ground-state on the $\mathcal{T}-$graph $\mathcal{T}_1$. Then,  there exist $\lambda_0>0$ for which
	\[
	\frac{\partial\Theta}{\partial\lambda}(6,\lambda)<0\quad\text{for}\quad \lambda\in(0,\lambda_0)\quad\text{for $\quad\lambda\in(\lambda_0,\infty)$}\quad\text{and}\quad 	\frac{\partial\Theta}{\partial\lambda}(6,\lambda)>0.
\]
\end{theorem}
We prove this result in a series of lemmas below. For now, having these results at hand, we conclude the proof of the first result of this paper.

\begin{proof}[Proof of Theorem \ref{thm:transition_particular}]
We apply Theorem \ref{thm:transition} (see also Remark \ref{rem:lambda_estrela}) to the curve of positive action ground-states $\lambda\in (0,\infty)\mapsto u^\lambda$. Indeed, condition $(A_1)$ is satisfied by Lemma \ref{lemma:A1}, $(A_2)$ by Lemma \ref{lemma:5.32} and $(A_3)$ (for $\lambda\neq \lambda^*$) is satisfied by Proposition \ref{prop:5.26}. Condition $(H)$, on the other hand, is satisfied by Theorem \ref{thm:description_edge_by_edge}. Recalling from Section \ref{sec:explicit_mass} that $\ell=\sqrt{\lambda}$ and
\[
\Theta(p,\lambda)=\lambda^\frac{6-p}{2(p-2)}\Theta_1(p,z(p,\ell)),
\]
the map $(p,\lambda)\mapsto \frac{\partial\Theta}{\partial \lambda}(p,\lambda)$ is continuous by Lemmas \ref{cor:4.2.Temp} and \ref{lemma:4.7.Temp}. Finally, the map $p\mapsto \Theta(6,\lambda)$ is not monotone by Theorems \ref{th:5.28} (for the tadpole) and \ref{th:5.29} (for the $\mathcal{T}$-graph).
\end{proof}

In the next section, we proof Theorem \ref{th:5.29}.
\subsection{Exact behavior of $\lambda\mapsto \Theta(6,\lambda)$ in the $\mathcal{T}$-graph}

From now on, we work in the $\mathcal{T}$-graph. Recall from \eqref{ThetaScaled} and \eqref{eq:Theta_1} that $\Theta(6,\lambda)=\Theta_1(6,\ell)$. From Lemma \ref{lemma:Theta_1_expressions} we write, with some abuse of notation,
\[
\Theta_1(6,\ell)=\Theta_1(6,z(p,\ell)).
\]
The derivative of $\Theta_1$ in $z$ is given in Lemma \ref{lemma:4.7.Temp} for a general $p$. By taking $p=6$, we deduce the following.

\begin{lemma}\label{lemma:5.34}
In the $\mathcal{T}$-graph, for $z\in \left(0,\sqrt[4]{3}\right)$ we have
	\begin{align}\label{MapMuDerivative}
		\sqrt{2}\frac{\partial \Theta_1}{\partial z}(6,z)(-3f(z)-f(1))&=-6z^2\sqrt{f(z)}+f'(z)\int_z^{f_2^{-1}(-3f(z))}\frac{1-t^2}{(1+t^2)^2}\frac{dt}{\sqrt{f(t)+3f(z)}}.
	\end{align}
\end{lemma}
\begin{proof}
	By taking $p=6$ in Lemma \ref{lemma:4.7.Temp}-1, we have
	\begin{align*}
		\sqrt{2}\frac{\partial \Theta_1}{\partial z}(6,z)&\left(-3f(z)-f(1)\right)=-6z^2\sqrt{f(z)}\\
		+&3f'(z)\int_z^{f_2^{-1}(-3f(z))}\left(-\frac{t^2}{2}-2t\frac{f(t)-f(1)}{f'(t)}+t^2f''(t)\frac{f(t)-f(1)}{f'(t)^2}\right)\frac{dt}{\sqrt{f(t)+3f(z)}}.
	\end{align*}
	Finally, by recalling that $f(t)=\frac{1}{2} t^2-\frac{1}{6}t^6$, we have 
\[
		-\frac{t^2}{2}-2t\frac{f(t)-f(1)}{f'(t)}+t^2f''(t)\frac{f(t)-f(1)}{f'(t)^2}=\frac{(1-t^2)}{3(1+t^2)^2}.\qedhere
\]
\end{proof}

In the following lemmas, we analyze the sign of $\frac{\partial\Theta_1}{\partial z}(6,z)$. We split this analysis in the cases $z\geq 1$ and $z\in(0,1)$, starting with the former.
\begin{lemma}\label{lemma:5.35}
Under the previous notations, in the $\mathcal{T}$-graph we have
	$\frac{\partial \Theta_1}{\partial z}(6,z)>0$ for any $z\in [1,\sqrt[4]{3})$.
\end{lemma}
\begin{proof}
	Suppose that $z\geq1$. Using the explict expression $f'(z)=z(1-z^4)$, integration by parts yields
	\begin{align*}
		\int_z^{f_2^{-1}(-3f(z))}\frac{1-t^2}{(1+t^2)^2}\frac{dt}{\sqrt{f(t)+3f(z)}}&=\int_z^{f_2^{-1}(-3f(z))}\frac{2}{t(1+t^2)^3}\frac{d}{dt}\sqrt{f(t)+3f(z)}dt\\
		&=-\frac{4\sqrt{f(z)}}{z(1+z^2)^3}+2\int_z^{f_2^{-1}(-3f(z))}\frac{1+7t^2}{t^2(1+t^2)^4}\sqrt{f(t)+3f(z)}dt.
	\end{align*}
	Thus,\vspace{-0.2cm}
	\begin{align*}
		\sqrt{2}\frac{\partial \Theta_1}{\partial z}(6,z)(-3f(z)-f(1))=&-6z^2\sqrt{f(z)}-\frac{4(1-z^2)\sqrt{f(z)}}{(1+z^2)^2}\\
		&+2f'(z)\int_z^{f_2^{-1}(-3f(z))}\frac{1+7t^2}{t^2(1+t^2)^4}\sqrt{f(t)+3f(z)}dt.
	\end{align*}
	Since $f'(z)\leq0$ for $z\geq1$, $(-3f(z)-f(1))<0$, and
	\[
	-6z^2\sqrt{f(z)}-\frac{4(1-z^2)\sqrt{f(z)}}{(1+z^2)^2}=-\frac{\sqrt{f(z)}}{(1+z^2)^2}\left(6z^2(1+z^2)^2+4(1-z^2)\right)<0,
	\]
	 we have $\frac{\partial \Theta_1}{\partial z}(6,z)>0$.
\end{proof}

In order to study $\frac{\partial \Theta_1}{\partial z}(6,z)$ for $z\in (0,1)$, recalling \eqref{MapMuDerivative}, we write
\[
\sqrt{2}\frac{\partial \Theta_1}{\partial z}(6,z)\frac{-3f(z)-f(1)}{f'(z)}=F(z)-G(z),
\]
where  $F,G:(0,1)\to\R$ are defined by the following expressions:
\begin{equation}\label{FunctionsFG}
	F(z)=\int_{z}^{f_2^{-1}(-3f(z))}\frac{1-t^2}{(1+t^2)^2}\frac{dt}{\sqrt{f(t)+3f(z)}}\quad\text{and}\quad G(z)=\frac{6z^2\sqrt{f(z)}}{f'(z)}.
\end{equation}

\begin{lemma}\label{lemma:5.36}
	Let $F,G:(0,1)\to\R$ be defined as in \eqref{FunctionsFG}. Then,\setlist{nolistsep}
	\begin{enumerate}[noitemsep]
		\item $F$ is decreasing in the interval $(0,1)$,
		\item $\lim\limits_{z\to0^+}F(z)=+\infty$ and $F(1):=\lim\limits_{z\to1^-}F(z)<0$,
		\item $G(z)$ is increasing in $z\in(0,1)$,
		\item $\lim\limits_{z\to0^+}G(z)=0$ and $\lim\limits_{z\to1^-}G(z)=+\infty$.
	\end{enumerate}
\end{lemma}

\begin{remark}The proof of Lemma \ref{lemma:5.36} is inspired by the proof of \cite[Lemma 10]{noja2020standing}, where the authors analyzed the same behaviour for the tadpole graph. In that work, to understand the sign of $F'$, the authors integrate by parts the expression of $F$ and then directly compute the derivative with respect to $z$. This argument can be seen in \cite[Lemma 10]{noja2020standing} after equation (4.33). In our case, however, if we integrate by parts the expression $F$ and then differentiate with respect to $z$, we encounter the problem that $3f'(z)>0$ (as opposed to $-\frac{3}{8}A'(z)<0$ in \cite[Lemma 10]{noja2020standing}) causes the integral term in the expression of $F'$ to have positive sign, instead of the desired negative sign. Thus, we take a different approach. We rewrite the integrand in $F$ in a suitable way and prove that it is decreasing in $z$.
\end{remark}

\begin{proof}[Proof of Lemma \ref{lemma:5.36}]
	Item 3. is a consequence of 
	\begin{align*}
		G'(z)&=\frac{6}{f'(z)^2}\left(\left(2z\sqrt{f(z)}+\frac{z^2f'(z)}{2\sqrt{f(z)}}\right)f'(z)-z^2\sqrt{f(z)}f''(z)\right)\\
		&=\frac{3}{f'(z)^2\sqrt{f(z)}}\left(4zf(z)f'(z)+z^2(f'(z))^2-2z^2f(z)f''(z)\right)\\
		&=\frac{3z^2}{f'(z)^2\sqrt{f(z)}}\left(2f(z)z^2\left(1+3z^4\right)+f'(z)^2\right)>0,
	\end{align*}	
	while Item 4 is straightforward.

		As for Item 2, from dominated convergence,
	\begin{align*}
		\lim_{z\to1^-}F(z)=\int_1^{f_2^{-1}(-3f(1))}\frac{1-t^2}{(1+t^2)^2}\frac{dt}{\sqrt{f(t)+3f(1)}}\leq 0
	\end{align*}	
	while the limit as $z\to 0^+$ is a consequence of the fact that  $\frac{1}{\sqrt{f(t)t}}\sim \frac{\sqrt{2}}{t}$ for $t$ close to 0 (not integrable) and $\frac{1}{\sqrt{f(t)}}\sim \frac{c}{\sqrt{t-\sqrt[4]{3}}}$ for $t$ close to $f_2^{-1}(0)=\sqrt[4]{3}$ (integrable).

	 Finally, we prove Item 1. We rewrite $F$ as
\[
		F(z)=\left(\int_z^1+\int_1^{f_2^{-1}(-3f(z))}\right)\frac{1-t^2}{(1+t^2)^2}\frac{dt}{\sqrt{f(t)+3f(z_0)}}=:F_1(z)+F_2(z).
\]
	Differentiating under the integral we have, for all $z\in(0,1)$,
	\begin{equation*}
		F'_1(z)=-\frac{1-z^2}{2(1+z^2)^2\sqrt{f(z)}}-\frac{3}{2}f'(z)\int_z^1\frac{1-t^2}{(1+t^2)^2}(f(t)+3f(z))^{-\frac{3}{2}}dt<0.
	\end{equation*}
For $F_2$, taking the variable change $t=1+sI_z$, where $I_z=f_2^{-1}(-3f(z))-1>0$, we obtain
	\begin{equation}\label{F2DecreasingForm}
		F_2(z)=\int_0^1\frac{(1-(1+sI_z)^2)I_zds}{(1+(1+sI_z)^2)^2\sqrt{f(1+sI_z)+3f(z)}}=:\int_0^1K(z,s)ds.
	\end{equation}
To conclude the proof, we claim that
the map $z\mapsto K(z,s)$ is strictly decreasing for every $s\in(0,1)$.
To simplify notation, we let $B(z,s)=f(1+sI_z)+3f(z)>0$ and $X(z,s)=1+sI_z>1$. Differentiating $K$ with respect to $z$ yields
	\begin{align}
		\frac{\partial K}{\partial z}(z,s)=&\frac{1}{(1+X(z,s)^2)^4B(z,s)}\left[(1+X(z,s)^2)^2\sqrt{B(z,s)}\left(-2X(z,s)sI'_zI_z+(1-X(z,s)^2)I'_z\right)\right.\nonumber\\
		&-\left.(1-X(z,s)^2)I_z\left(4sX(z,s)(1+X(z,s)^2)I'_z\sqrt{B(z,s)}+(1+X(z,s)^2)^2\frac{\partial_z B(z,s)}{2\sqrt{B(z,s)}}\right)\right]\nonumber\\
		=&\frac{1}{2(1+X(z,s)^2)^4B(z,s)^\frac{3}{2}}\left[2(1+X(z,s)^2)^2B(z,s)I'_z\left(-2X(z,s)sI_z+(1-X(z,s)^2)\right)\right.\nonumber\\
		&-\left.(1-X(z,s)^2)I_z\left(8sX(z,s)(1+X(z,s)^2)I'_zB(z,s)+(1+X(z,s)^2)^2 \partial_z B(z,s)\right)\right]. \nonumber \\
		=&\frac{1}{2(1+X(z,s)^2)^4B(z,s)^\frac{3}{2}}\left[Q(z,s) + 2(1+X(z,s)^2)^2B(z,s)I'_z(1-X(z,s)^2)\right. \nonumber \\
		&\left.-(1-X(z,s)^2)(1+X(z,s)^2)^2I_z\partial_zB(z,s) \right], \label{DerK}
	\end{align}
	for 
\[
Q(z,s):=-(1-X(z,s)^2)I_z8sX(z,s)(1+X(z,s)^2)I'_zB(z,s)-4(1+X(z,s)^2)^2B(z,s)I'_zX(z,s)sI_z.
\]
	Recall now the following simple facts: for $z\in(0,1)$, we have  $X(z,s)>1$, $f'(z)>0$, $I_z=f_2^{-1}(-3f(z))-1>0$ and 
	\[
	I'_z=\frac{-3f'(z)}{f'(f_2^{-1}(-3f(z)))}>0,\ \text{for all}\ z\in(0,1),\quad \text{ with }\quad	I'_0=0.
	\]
Observe now that the term outside the brackets in \eqref{DerK} is positive. We show that the term inside are negative:
	\begin{itemize}
	\item $2(1+X(z,s)^2)^2B(z,s)I'_z(1-X(z,s)^2)<0$, by the observations made in the previous paragraph.
	\item Observe that $-(1-X(z,s)^2)(1+X(z,s)^2)^2I_z>0$ and 	\begin{equation*}
\frac{\partial^2 B}{\partial s\partial z}(z,s)=\frac{\partial}{\partial s}\left(f'(1+sI_z)sI'_z+3f'(z)\right)=f''(1+sI_z)s^2{I'_z}^2+f'(1+sI_z)I'_z<0.
	\end{equation*}
	Thus, the map $s\mapsto \partial_z B(z,s)$ is strictly decreasing for every $z\in(0,1)$. Moreover, since $\partial_zB(0,s)=f'(1)sI'_0+3f'(0)=0$, we have that $\partial_z B(z,s)<0$ for all $(z,s)\in(0,1)^2$. This implies that the term $(1+X(z,s)^2)^2\partial_z B(z,s)$ in \eqref{DerK} is negative.
\item Finally, we show that $Q(z,s)$ has negative sign.

 We note that $Q$ can be simplified to
	$$Q(z,s)=-4sI_zI'_zB(z,s)X(z,s)(1+X(z,s)^2)(3-X(z,s)^2).$$
	The term $-4sI_zI'_zB(z,s)X(z,s)(1+X(z,s)^2)$ is negative. We show that $h(z,s):=3-X(z,s)^2>0$. Differentiating $h$ with respect to $s$, we get
	$\frac{\partial h}{\partial s}(z,s)=-2X(z,s)I_z<0$; that is, $s\mapsto h(z,s)$ is strictly decreasing for every $z\in(0,1)$. To conclude, we  show that $h(z,1)>0$. Note that
	$$h(z,1)=3-X(z,1)^2=3-(1+I_z)^2=3-(f_2^{-1}(-3f(z)))^2,$$
	and that (recalling that $f(1)=1/3$, since $p=6$), 
	$$\frac{\partial h}{\partial z}(z,1)=\frac{6f'(z)(f_2^{-1}(-3f(z))}{f'(f_2^{-1}(-3f(z)))}<0\quad \text{ and }\quad h(1,1)=3-(f_2^{-1}(-3f(1)))^2=3-(f_2^{-1}(-1))^2>0$$
	\end{itemize}
	This finishes the proof of Item 1.
\end{proof}

We now conclude this section with the following.

\begin{proof}[Proof of Theorem \ref{th:5.29}]
	By Lemmas \ref{lemma:5.35} and \ref{lemma:5.36},  there exists a unique $z_0\in(0,1)$ such that 
	\begin{equation}\label{existence_of_z0}
	\frac{\partial \Theta_1}{\partial z}(6,z)<0\quad \text{ for any } z\in(0,z_0),\qquad  \frac{\partial \Theta_1}{\partial z}(6,z)>0\quad \text{for any $z\in (z_0,\sqrt[4]{3})$. }
	\end{equation}
	The conclusion now follows from
	\[
	\frac{\partial \Theta}{\partial \lambda}(6,\lambda)=	\frac{\partial }{\partial \lambda}(\Theta_1(6,\ell(\lambda)))=\frac{\partial \Theta_1}{\partial z}(6,z(6,\ell(\lambda)))\frac{\partial z}{\partial\ell}(6,\ell(\lambda))\frac{d \ell}{d\lambda}(\lambda). 
	\]
	The conclusion now follows from \eqref{existence_of_z0}, Lemma \ref{cor:4.2.Temp} and the explicit expression $\ell(\lambda)=\lambda^{\frac{1}{2}}$.
\end{proof}

\section{Asymptotic stability analysis in $\lambda$ and $p$}\label{sec:asymptotics}

Recall from \eqref{ThetaScaled} that
\begin{equation}\label{eq:alpha}
    \Theta(p,\lambda)=\lambda^\alpha\Theta_1(p,\ell(\lambda)),\quad \text{where $\alpha=\frac{6-p}{2(p-2)}$ and $\ell(\lambda)=\sqrt{\lambda}$},
\end{equation}
and $\Theta_1$ is as in \eqref{Theta_1_T'} and \eqref{Theta_1_TADPOLE'}, for the $\mathcal{T}$-graph and tadpole graph, respectively. Setting
\begin{equation}\label{z_func_ell}
z(p,\ell)=L^{-1}(p,\ell)
\end{equation}
(cf. Theorem \ref{thm:description_edge_by_edge}), Lemma \ref{cor:4.2.Temp} implies
\begin{align}
     \frac{\partial\Theta}{\partial\lambda}(p,\lambda)&=\lambda^{\alpha-1}\left(\alpha\Theta_1(p,z(p,\ell(\lambda))+\frac{\sqrt{\lambda}}{2}\frac{\partial\Theta_1}{\partial z}(p,z(p,\ell(\lambda))\frac{\partial z}{\partial \ell}(p,\ell(\lambda))\right)\nonumber \\
     &=\lambda^{\alpha-1}\left(\alpha\Theta_1(p,z(p,\ell(\lambda))+\frac{L(p,z)}{2}\frac{\partial\Theta_1}{\partial z}(p,z(p,\ell(\lambda))\frac{1}{\frac{\partial L}{\partial z}(p,z(p,\ell(\lambda)))}\right)\label{der_Theta_z}.
\end{align}
Since $L(p,z(p,\ell))=\ell=\sqrt{\lambda}$, we study the asymptotic behavior of $\frac{\partial\Theta}{\partial\lambda}(p,\lambda)$ in $\lambda$ and $p$ in terms of the variable $z$ by analyzing the behavior of 
\begin{equation}\label{der_Theta_z2}
   z\mapsto \alpha\Theta_1(p,z)+\frac{L(p,z)}{2}\frac{\partial\Theta_1}{\partial z}(p,z)\frac{1}{\frac{\partial L}{\partial z}(p,z)}.
\end{equation}
\subsection{Asymptotics in $\lambda$}\label{sec:asymptotics_lambda}
Observe that, by \eqref{eq:asymptotic_L}, \eqref{z_func_ell} and the behavior of $\varphi$,
\begin{equation}\label{eq:limitzlambda}
    \lambda\to 0^+ \iff z\to \varphi(0)^- \quad \text{ and } \quad \lambda\to +\infty \iff z\to 0^+.
\end{equation}

\begin{lemma}\label{prop6.1}
    We have 
    \[
    \lim_{z\to 0^+} \Theta_1(p,z)=\begin{cases} \frac{1}{2}\|\varphi\|^2_{L^2(\R)} & \text{ in the } \mathcal{T}\text{-graph},\\
    \|\varphi\|^2_{L^2(\R)} & \text{ in the tadpole graph,}
    \end{cases}
    \]
    and
    \[
    \lim_{z\to \varphi(0)^-} \Theta_1(p,z)=\begin{cases} \|\varphi\|^2_{L^2(\R)} & \text{ in the } \mathcal{T}\text{-graph},\\
    \frac{1}{2}\|\varphi\|^2_{L^2(\R)} & \text{ in the tadpole graph.}
    \end{cases}
    \]   
\end{lemma}
\begin{proof}
Recalling the decompositions \eqref{Theta_1_T'} and \eqref{Theta_1_TADPOLE'}, the result is a direct consequence of the fact that
\begin{equation}\label{mu_functions}
  \lim_{z\to 0^+}  \mu_1(p,z,\theta)=\frac{1}{\sqrt{2}}\int_{0}^{(p/2)^\frac{1}{p-2}}\frac{t^2dt}{\sqrt{f(t)}}=\frac{1}{2}\|\varphi\|_{L^2(\R)^2},\qquad  \lim_{z\to \varphi(0)^-}  \mu_1(p,z,\theta)=0,
\end{equation}    
and, recalling that $\varphi(0)=(p/2)^\frac{1}{p-2}$,
\[
\lim_{z\to 0^+}\mu_2(p,z)=0,\qquad \lim_{z\to \varphi(0)^-}\mu_2(p,z)=\frac{1}{\sqrt{2}}\int_0^{\varphi(0)}\frac{t^2dt}{\sqrt{f(t)}}=\frac{1}{2}\|\varphi\|_{L^2(\R)^2}.\qedhere
\]
\end{proof}

\begin{lemma}\label{assymptotic_L1}
    Fix $p>2$ with $p\neq 6$ and $\theta>0$. Then there exist constants $C_1=C_1(p,\theta)>1$ and $C_2=C_2(p,\theta)<1$ such that
    \[
\frac{1}{C_1}|\ln (z)|\leq L(p,z) \leq C_2|\ln (z)|\quad \text{ and } \quad  -\frac{1}{C_2z}\leq \frac{\partial L}{\partial z}(p,z)\leq -\frac{C_2}{z}\qquad \text{ for } z\sim 0^+,
    \]
while
     \[
 L(p,z) \sim \frac{2\theta}{\sqrt{|f'(\varphi(0))|}}\sqrt{\varphi(0)-z}\quad \text{ and }\quad  \frac{\partial L}{\partial z}(p,z) \sim -\frac{\theta}{2|f'(\varphi(0))|\sqrt{\varphi(0)-z}}\qquad \text{ for } z\sim \varphi(0)^-.
    \]
\end{lemma}

\begin{proof}
\emph{Step 1.} Behavior of $L(p,z)$ as $z\to 0^+$.\newline

From \eqref{eq:Length}, 
\[
L(p,z)=\frac{1}{\sqrt{2}} \int_z^{f_{2}^{-1}\left(\left(1-\theta^2\right) f(z)\right)} \frac{d t}{\sqrt{f(t)-\left(1-\theta^2\right) f(z)}}= \frac{1}{\sqrt{2}}\left(\int_z^1 +\int_1^{f_{2}^{-1}\left(\left(1-\theta^2\right) f(z)\right)}\right)\frac{d t}{\sqrt{f(t)-\left(1-\theta^2\right) f(z)}}.
\]
Since $\varphi(0)$ is a simple root of $f$, $\frac{1}{\sqrt{f(z)}}$ is integrable close to $t=\varphi(0)$, and
\[
\int_1^{f_{2}^{-1}\left(\left(1-\theta^2\right) f(z)\right)}\frac{d t}{\sqrt{f(t)-\left(1-\theta^2\right) f(z)}}\to \int_1^{\varphi(0)}\frac{1}{\sqrt{f(t)}}\, dt<\infty\qquad \text{ as } t\to \varphi(0)^-.
\]
On the other hand, since $f$ is increasing in $(0,1)$, for $1>t>z$ we have
\begin{equation}\label{eq:bounded_above_below}
\begin{cases}
\theta\sqrt{f(t)}\leq \sqrt{f(t)-(1-\theta^2)f(z)}\leq \sqrt{f(t)}& \text{ if } \theta^2\leq 1,\\ \max\{{\theta\sqrt{f(z)}},\sqrt{f(t)}\}\leq \sqrt{f(t)-(1-\theta^2)f(z)}\leq \theta\sqrt{f(t)}& \text{ if } \theta^2\geq 1.
\end{cases}
\end{equation}
Therefore, as $z\to 0^+$, $\int_z^{1} \frac{d t}{\sqrt{f(t)-\left(1-\theta^2\right) f(z)}}$ is bounded from above and below by a multiple of
\[
\int_z^1 \frac{1}{\sqrt{f(t)}}\, dt\sim \sqrt{2}\int_z^1 \frac{1}{t}\, dt\sim -\sqrt{2}\ln z.
\]
From this, we deduce the existence of $C>0$ such that
\[
\frac{1}{C}|\ln z|\leq L(p,z)\leq C|\ln z| \qquad \text{ as } z\to 0^+.
\]

\smallbreak

\noindent \emph{Step 2.} Behavior of $L(p,z)$  as $z\to \varphi(0)^-$.\newline

As $z\to \varphi(0)^-$, since $\varphi(0)$ is a simple zero of $f$, $\frac{1}{\sqrt{f(t)}}$ is integrable near $t=\varphi(0)$, and 
\[
f(t)\sim f'(\varphi(0))(t-\varphi(0)),\qquad f_2^{-1}((1-\theta^2)f(t))\sim \varphi(0)+(1-\theta^2)(t-\varphi(0)),
\]we have\footnote{To properly justify the asymptotic estimate, perform a change of variable $t=z+s(f_2^{-1}((1-\theta^2)f(z))-z)$, and use the dominated convergence theorem.}

\begin{align*}
L(p,z)&\sim \frac{1}{\sqrt{2}}\int_z^{ \varphi(0)+(1-\theta^2)(z-\varphi(0))} \frac{dt}{\sqrt{f'(\varphi(0))(t-\varphi(0))-(1-\theta^2)f'(\varphi(0))(z-\varphi(0))}}\, dt\\
 &=\frac{1}{\sqrt{2|f'(\varphi(0))|}}\int_z^{ \varphi(0)+(1-\theta^2)(z-\varphi(0))} \frac{dt}{\sqrt{\varphi(0)+(1-\theta^2)(z-\varphi(0))-t}}\, dt\\
 &=\frac{2}{\sqrt{2|f'(\varphi(0))|}}\left[ \sqrt{\varphi(0)+(1-\theta^2)(z-\varphi(0))-t}\,dt\right]_z^{ \varphi(0)+(1-\theta^2)(z-\varphi(0))}=\frac{\sqrt{2}\theta}{\sqrt{|f'(\varphi(0))|}}\sqrt{\varphi(0)-z}.
\end{align*}

\bigbreak 

\noindent \emph{Step 3.} Behavior of $\frac{\partial L}{\partial z}(p,z)$ as $z\to 0^+$.\newline

Recall now from Lemma \ref{lemma:L'} that
\begin{equation}\label{eq:aux_asym1}
\sqrt{2}\frac{\partial L}{\partial z}(p,z)\left((1-\theta^2)f(z)-f(1)\right)=\theta\frac{f(1)}{\sqrt{f(z)}}+(1-\theta^2)f'(z)\int_z^{f_{2}^{-1}\left((1-\theta^2)f(z)\right)}\frac{A(t)dt}{\sqrt{f(t)-(1-\theta^2)f(z)}}
\end{equation}
    As $z\to 0^+$, we have
    \begin{equation*}\label{eq:aux_asym2}
   \sqrt{2}\left((1-\theta^2)f(z)-f(1)\right)\to -\sqrt{2}f(1) \quad \text{ and } \quad \theta\frac{f(1)}{\sqrt{f(z)}}\sim \frac{\theta\sqrt{2}f(1)}{z}.
    \end{equation*}
As for the integral term, since $\varphi(0)$ is a simple root of $f$ and $A(t)\sim f(1)/t^2$,
\[
(1-\theta^2)f'(z)\int_z^{f_{2}^{-1}\left((1-\theta^2)f(z)\right)}\frac{A(t)dt}{\sqrt{f(t)-(1-\theta^2)f(z)}}\sim (1-\theta^2)z\int_z^1 \frac{f(1)}{t^2\sqrt{f(t)-(1-\theta^2)f(z)}}\, dt.
\]
Recalling \eqref{eq:bounded_above_below}, and using
\[
z\int_z^1 \frac{f(1)}{t^2\sqrt{f(t)}}\, dt,\ \frac{zf(1)}{\sqrt{f(z)}}\int_z^1 \frac{dt}{t^2}\sim \frac{\sqrt{2}f(1)}{z},
\]
in \eqref{eq:aux_asym1}, we deduce that
\begin{equation*} 
\frac{\sqrt{2}f(1)}{z}(\theta+1-\theta^2)\lesssim -\sqrt{2}f(1)\frac{\partial L}{\partial z}\leq \frac{\sqrt{2}f(1)}{\theta z} \qquad \text{ if } \theta<1,
\end{equation*}
and
\begin{equation*} 
\frac{\sqrt{2}f(1)}{\theta z}\lesssim -\sqrt{2}f(1)\frac{\partial L}{\partial z}\lesssim \frac{\sqrt{2}f(1)}{\theta z} \qquad \text{ if } \theta\geq 1.
\end{equation*}

\noindent{ \textit{Step 4}.} Behavior of $\frac{\partial L}{\partial z}$ as $z\sim \varphi(0)^-$.\newline

This time,
since
\[
A(t)\sim \frac{1}{2}+\frac{f(1)f''(\varphi(0))}{f'(\varphi(0))^2}=:A_0
\]
we have from \eqref{eq:aux_asym1} that
\begin{align*}
-f(1)\sqrt{2}\frac{\partial L}{\partial z}(p,z) &\sim\frac{\theta f(1)}{\sqrt{f(\varphi(0))(z-\varphi(0))}}+\int_z^{\varphi(0)+(1-\theta^2)(z-\varphi(0))}\frac{A_0 (1-\theta^2)\sqrt{|f'(\varphi(0))|}(\varphi(0)-z)}{\sqrt{\varphi(0)+(1-\theta^2)(z-\varphi(0))-t}}\, dt\\
&\sim \frac{\theta f(1)}{\sqrt{f(\varphi(0))(z-\varphi(0))}},
\end{align*}
as wanted. 
\end{proof}

\begin{lemma}\label{assymptotic_mu_der}
    Fix $p>2$ with $p\neq 6$ and $\theta\in\R$. Then, there exists a constant $C=C(p,\theta)$ such that
    \begin{equation}\label{eq:d_mu_1}
  \frac{1}{C}z|\ln z|  \leq \frac{\partial\mu_1}{\partial z}(p,z,\theta)\leq C z|\ln z| \qquad \text{ for } z\sim 0^+,
    \end{equation}
    \[
        \frac{\partial\mu_1}{\partial z}(p,z,\theta)\sim -\frac{\theta \varphi(0)^2}{\sqrt{2|f'(\varphi(0))|}}\frac{1}{\sqrt{\varphi(0)-z}} \qquad \text{ for } z\sim \varphi(0)^-,
    \]
        and
     \begin{equation}\label{eq:asym_derivative_mu_2}
        \frac{\partial\mu_2}{\partial z}(p,z)\sim\begin{cases}
            z,\ & \text{for}\ z\sim0^+,\\
           \frac{\varphi(0)^2}{\sqrt{2|f'(\varphi(0))|}}\frac{1}{\sqrt{(\varphi(0)-z)}},& \text{for}\ z\sim\varphi(0)^-.
        \end{cases}
    \end{equation}
\end{lemma}

\begin{proof} \emph{Step 1.} Estimates for $\frac{\partial\mu_2}{\partial z}(p,z)$.\newline

    The estimates \eqref{eq:asym_derivative_mu_2} follow directly from the formula (see Lemma \ref{lemma:mu_1reg})
    $$ 
    \frac{\partial\mu_2}{\partial z}(p,z)= \frac{z^2}{\sqrt{2}\sqrt{f(z)}}.
    $$

\smallbreak

\noindent \emph{Step 2.} Estimates for $\frac{\partial\mu_1}{\partial z}(p,z,\theta)$ as $z\to 0^+$.\newline

Recall also from Lemma \ref{lemma:mu_1reg} that
\begin{align}
		\sqrt{2}\frac{\partial\mu_1}{\partial z}(p,z,\theta)&((1-\theta^2)f(z)-f(1))=\theta z^2\frac{f(1)}{\sqrt{f(z)}}\nonumber\\
        &-(1-\theta^2)f'(z)\int_z^{f_{2}^{-1}((1-\theta^2)f(z))}\left(t^2-\frac{g(t)}{2}\right)\frac{dt}{\sqrt{f(t)-(1-\theta^2)f(z)}},\label{mu1DerAux}
	\end{align}
where for every $p>2$, the function $g:\R\to\R$ is given by
	$$
    g(t)=3t^2-2t^2\frac{(f(t)-f(1))}{f'(t)^2}f''(t)+4t\frac{(f(t)-f(1))}{f'(t)}.
    $$
Since $t^2-g(t)/2\sim f(1)$, we have
\begin{equation}\label{eq:auxiliary_asymptotic_mu_2}
-\sqrt{2}f(1)\frac{\partial \mu_1}{\partial \theta}(p,z,\theta)\sim \sqrt{2}\theta f(1)z-(1-\theta^2)z\int_z^1 \frac{f(1)}{\sqrt{f(t)+(1-\theta^2)f(z)}}\, dt
\end{equation}
As
\begin{equation}\label{eq:bounded_above_below2}
\begin{cases}
    \theta\sqrt{f(t)}\leq \sqrt{f(t)-(1-\theta^2)f(z)}\leq \sqrt{f(t)}\ \text{ if } \theta^2\leq 1,\\ {\sqrt{f(t)}}\leq \sqrt{f(t)-(1-\theta^2)f(z)}\leq \theta\sqrt{f(t)}\ \text{ if } \theta^2\geq 1
\end{cases}
\end{equation}
for $0<z<t$, the integral term in \eqref{eq:auxiliary_asymptotic_mu_2} is bounded from above and below by a multiple of
\[
z\int_z^1\frac{1}{\sqrt{t}}\, dt\sim z\int_z^1 \frac{\sqrt{2}}{t}\sim -\sqrt{2}z\ln z,
\]
from which \eqref{eq:d_mu_1} follows.

\smallbreak

\noindent \emph{Step 3.} Estimates for $\frac{\partial\mu_1}{\partial z}(p,z,\theta)$ as $z\to \varphi(0)^-$.

This time,
\[
t^2-\frac{g(t)}{2}\sim \varphi(0)^2\to -\frac{\varphi(0)^2}{2}-2\varphi(0)^2f''(\varphi(0))\frac{f(1)}{f'(\varphi(0))^2}+4\varphi(0)\frac{f(1)}{f'(\varphi(t))}=:g_0,
\]
so
\begin{align*}
-\sqrt{2}f(1)\frac{\partial\mu_1}{\partial z}(p,z,\theta)&\sim\theta \varphi(0)^2\frac{f(1)}{\sqrt{f'(\varphi(0))(z-\varphi(0))}}
        -\int_z^{\varphi(0)+(1-\theta^2)(z-\varphi(0))}\frac{(1-\theta^2)g_0f'(\varphi(0))(z-\varphi(0)) dt}{\sqrt{\varphi(0)+(1-\theta^2)(z-\varphi(0))-t}},\\
\end{align*}
as wanted.
\end{proof}

\begin{remark}
    Through a simple (but slightly cumbersome) analysis of the above proofs, one can check that the asymtotic estimates of Lemmas \ref{prop6.1}-\ref{assymptotic_mu_der} are uniform for $p\in [2+\delta,6-\delta]$ and $p\in[6+\delta,1/\delta]$, for any $\delta>0$.
\end{remark}

\begin{proof}[Proof of Theorem \ref{thm:asymptotics}-1.,2]
   For $z\sim0^+$, Lemma \ref{assymptotic_mu_der}, together with \eqref{Theta_1_T'} and \eqref{Theta_1_TADPOLE'}, implies that
\[
\left| \frac{\partial \Theta_1}{\partial z}(p,z)\right|\leq Cz|\ln z| \quad \text{ for } z\sim 0^+,\qquad 
\left| \frac{\partial \Theta_1}{\partial z}(p,z)\right|\leq C\frac{1}{\sqrt{\varphi(0)-z}} \quad \text{ for } z\sim \varphi(0)^-.
\]
By Lemma \ref{assymptotic_L1},
    \[
    \left|\frac{L(p,z)}{2}\frac{\partial\Theta_1}{\partial z}(p,z)\frac{1}{\frac{\partial L}{\partial z}(p,z)}\right|\leq 
    \begin{cases}
    Cz^2|\ln z|^2& \text{ for } z\sim 0^+,\\ 
    C\sqrt{\varphi(0)-z}& \text{ for } z\sim \phi(0)^-
    \end{cases}.
    \]
    Since $\Theta_1$ converges to a positive constant (Lemma \ref{prop6.1}), by \eqref{der_Theta_z}, the sign of $\frac{\partial \Theta}{\partial \lambda}$ is determined by $\alpha=\frac{6-p}{2(p-2)}$. In particular, it follows from \eqref{eq:limitzlambda} that, for $p>6$,
    \[
    \frac{\partial \Theta}{\partial \lambda}(p,z)<0  \text{ for }\lambda\sim 0^+,\infty.
    \]
    and, for $p<6$,
     \[
    \frac{\partial \Theta}{\partial \lambda}(p,z)>0  \text{ for }\lambda\sim 0^+,\infty.
    \]
    The conclusion now follows from Theorem \ref{th:5.27}, since conditions $(A_1)$, $(A_2)$ and $(A_3)$ were already verified for the $\mathcal{T}$ and tadpole graphs in Section \ref{sec:GSS_particular}.
\end{proof}

\subsection{Asymptotics in $p$}\label{sec:asymptotics_p}
As the arguments are different, we split this section into two parts, dealing first with the case $p\sim 2^+$ and then with the case $p\sim \infty$.

\subsubsection{Asymptotics for $p\sim 2^+$}

Throughout this section, it is important to observe that, for $\alpha$ given by \eqref{eq:alpha},
\[
\lim_{p\to 2^+}\left(\frac{p}{2}\right)^\frac{1}{p-2}= \sqrt{e}\qquad \text{ and} \qquad \lim_{p\to 2^+} \alpha=\lim_{p\to 2^+}\frac{p-6}{2(p-2)}=\infty.
\]
\begin{lemma}\label{lemma:asympt_p_aux1}
    Take $\delta>0$. Define
    \[
    f_\infty(z):=\frac{z^2}{4}(1-2\ln z),
    \]
    and let $f_{2,\infty}^{-1}$ denote the inverse of $f_\infty|_{[1,\infty)}$.
    Then, as $p\sim 2^+$,
    \begin{equation}\label{eq:asympt_p=2}
        f(z)\sim f_\infty(z)(p-2),\quad f'(z) \sim f'_\infty(z) (p-2),\quad f''(z)\sim f''_\infty(z)(p-2)\quad f'''(z)\sim f'''_\infty(z)(p-2)
    \end{equation}
    and
    \[
    f_2^{-1}((1-\theta^2) f(z))\sim f_{2,\infty}^{-1}((1-\theta^2)f_\infty(z))
    \]
    uniformly for $z\in [\delta,1/\delta]$.
\end{lemma}

\begin{proof}[Proof of Lemma \ref{lemma:asympt_p_aux1}]
    The first three asymptotics follow directly by a Taylor expansion of $f(z)=\frac{z^2}{2}-\frac{z^p}{p}$ with Lagrange remainder. As for the last one, we claim that
 \begin{equation}\label{claimp2}
         f_{2}^{-1}((p-2)\cdot )\to f_{\infty}^{-1}(\cdot) \text{ uniformly in compact sets of } (0,\infty).
 \end{equation}
    Take $z_p:=f_{2}^{-1}((p-2) y)$. First, we observe that $z_p$ is bounded: otherwise, for each $L>0$, we would have (up to a subsequence) $z_p\geq z_0$ and, since $f|_{(1,\infty)}$ is decreasing,
    \[
    y=\frac{f(z_p)}{p-2}\leq \frac{f(L)}{p-2}\to f_\infty(z_0) \qquad \text{ as } p\to 2^+.
    \]
    Since $z_0$ is arbitrary, $y$ varies in a compact set, and $f_\infty(z_0)\to -\infty$ as $L\to \infty$, we obtain a contradiction. 
    
    Having shown that $z_p\ge 1$ is bounded, by the first estimate in \eqref{eq:asympt_p=2},
    \[
    y=\frac{f(z_p)}{p-2}=\frac{f(z_p)}{p-2}-f_\infty(z_p)+f_\infty(z_p)=f_\infty(z_p)+o(1),
    \]
    from which the claim \eqref{claimp2} follows.
    
    Finally, by \eqref{eq:asympt_p=2} and \eqref{claimp2},
\[
    f_2^{-1}((1-\theta^2) f(z))\sim f_2^{-1}((1-\theta^2)f_\infty(z)(p-2)+o(p-2))\sim f_{2,\infty}^{-1}((1-\theta^2)f_\infty(z)).\qedhere
\]
\end{proof}

\begin{lemma}\label{lem:lim_z_p2}
    Take $\delta>0$ small. Then, as $p\to 2^+$, uniformly in $z\in (\delta,\sqrt{e}-\delta)$ we have
\begin{equation}\label{eq:asympt_L_p2}
    \sqrt{p-2}\ L(p,z)\sim\frac{1}{\sqrt{2}} \int_z^{f_{2,\infty}^{-1}\left(\left(1-\theta^2\right) f_\infty(z)\right)} \frac{d t}{\sqrt{f_\infty(t)-\left(1-\theta^2\right) f_\infty(z)}}.
\end{equation}
In particular, given $\delta>0$, if $z(p,\lambda)$ is the unique solution to $L(p,z)=\lambda$, then
\begin{equation}\label{eq:lim_z}
    \lim_{p\to2^+}z(p,\lambda)=\sqrt{e},\quad \text{uniformly in }\lambda\in [\delta,1/\delta].
\end{equation}
\end{lemma}
\begin{proof}
    The asymptotic estimate of $L(p,z)$ is a direct consequence of the definition of $L$, formula \eqref{eq:Length}, and Lemma \ref{lemma:asympt_p_aux1}.

    By contradiction, suppose that \eqref{eq:lim_z} does not hold. Since $z(p,\lambda)<(p/2)^{1/(p-2)}\to\sqrt{e}$, there would exist a sequence $(p_n,\lambda_n)$, with $p_n\to 2^+$ and $\lambda_n\in [\delta,1/\delta]$, such that $z(p_n,\lambda_n)<\sqrt{e}-\delta'$, for some $\delta'>0$. In particular, by \eqref{eq:asympt_L_p2} and since the map $z\mapsto L(p,z)$ is strictly  decreasing for each $p>2$,
\begin{align*}
        \sqrt{\lambda_n}=L(p_n,z(p_n,\lambda_n)) &> L(p_n, \sqrt{e}-\delta') \\&\sim \frac{1}{\sqrt{2(p_n-2)}} \int_{\sqrt{e}-\delta'}^{f_{2,\infty}^{-1}\left(\left(1-\theta^2\right) f_\infty(\sqrt{e}-\delta')\right)} \frac{d t}{\sqrt{f_\infty(t)-\left(1-\theta^2\right) f_\infty(\sqrt{e}-\delta')}} \to \infty,
\end{align*}
     which is a contradiction.
\end{proof}

\begin{lemma} \label{lem:lim_p2}
Fix $\delta>0$. Given $\lambda\in [\delta,1/\delta]$, let $z(p,\lambda)$ be the unique solution to $L(p,z)=\sqrt{\lambda}$. Then, as $p\sim 2^+$, uniformly in $\lambda\in[\delta,1/\delta],$
\begin{equation*}
    \sqrt{p-2}\ \mu_1(p,z(p,\lambda),\theta)\sim 0, \quad \sqrt{p-2}\frac{\partial \mu_1}{\partial z}(p,z(p,\lambda),\theta) \sim -\frac{\theta z^2}{\sqrt{2f_\infty(z)}},
\end{equation*}
\begin{equation*}
    \sqrt{p-2}\ \mu_2(p,z(p, \lambda))\sim\frac{1}{\sqrt{2}}\int_0^{\sqrt{e}}\frac{t^2dt}{\sqrt{f_\infty(t)}},\quad \sqrt{p-2}\ \frac{\partial \mu_2}{\partial z}(p,z(p, \lambda))\sim \frac{z^2}{\sqrt{2f_\infty(z)}},
\end{equation*}
and
\begin{equation*}
    \sqrt{p-2}\frac{\partial L}{\partial z}(p, z(p,\lambda)) \sim -\frac{\theta}{\sqrt{2f_\infty(z)}}.
\end{equation*}
\end{lemma}
\begin{proof}
    By Lemmas \ref{lemma:asympt_p_aux1} and \ref{lem:lim_z_p2}, we have
    \[
z(p,\lambda)\sim \sqrt{e},\quad f_2^{-1}((1-\theta^2)f(z(p,\lambda)))\sim \sqrt{e}.
    \]
    On the other hand, the mapping
    \[
    t\mapsto f(t)-(1-\theta^2)f(z(p,\lambda)),\qquad t\in [z(p,\lambda), f_2^{-1}((1-\theta^2)f(z(p,\lambda)))],
    \]
    has a simple root at $t=f_2^{-1}((1-\theta^2)f(z(p,\lambda)))$. In particular, the map $t\mapsto s=f(t)-(1-\theta^2)f(z(p,\lambda))$ is invertible near $t=f_2^{-1}((1-\theta^2)f(z(p,\lambda)))$ and
 \begin{align*}
     &\int_z^{f_2^{-1}((1-\theta^2)f(z(p,\lambda)))}\frac{\text{bounded in }t}{\sqrt{f(t)-(1-\theta^2)f(z(p,\lambda))}}dt \\= &\int_0^{f(z(p,\lambda))-(1-\theta^2)f(z(p,\lambda))}\frac{\text{bounded in }s}{\sqrt{s}}\frac{ds}{f'(t(s))} \to 0\quad \text{ as }p\to 2^+.
 \end{align*}
    Taking into account this consideration, it now suffices to pass to the limit in the expressions of Lemmas \ref{lemma:Theta_1_expressions}, \ref{lemma:mu_1reg} and \eqref{H_p(z,theta)NotSingular}.
\end{proof}

\begin{proposition}
    Under the notations of Lemma \ref{lem:lim_p2}, as $p\sim 2^+$,
    \begin{equation}\label{eq:sinalp2}
         \frac{\partial \Theta}{\partial \lambda}(p,\lambda)=\alpha\Theta_1(p,z(p,\lambda))+\frac{L(p,z(p,\lambda))}{2}\frac{\partial\Theta_1}{\partial z}(p,z(p,\lambda))\frac{1}{\frac{\partial L}{\partial z}(p,z(p,\lambda))}>0
    \end{equation}
     uniformly in $\lambda\in[\delta,1/\delta]$.
\end{proposition}
\begin{proof}
    By Lemma \ref{lem:lim_p2}, and recalling from Lemma \ref{lemma:Theta_1_expressions} the relation between $\Theta_1$, $\mu_1$ and $\mu_2$ in the $\mathcal{T}$ and tadpole graphs, up to positive factors we have
    \[\alpha\Theta_1(p,z(p,\lambda)) \sim \frac{1}{\sqrt{2}(p-2)^{3/2}}\int_0^{\sqrt{e}} \frac{t^2dt}{\sqrt{f_\infty(t)}},\]
    \[
    \sqrt{p-2}\sqrt{f_\infty(z(p,\lambda))}\frac{\partial\Theta_1}{\partial z}(p,z(p,\lambda))\sim 0
    \]
    and
    \[
    \sqrt{p-2}\sqrt{f_\infty(z(p,\lambda))}\frac{\partial L}{\partial z}(p,z(p,\lambda))\sim -1.
    \]
    Since $L(p,z(p,\lambda))=\sqrt{\lambda}$, we find
    \begin{align*}
        \frac{\partial \Theta}{\partial \lambda}(p,\lambda)&=\alpha\Theta_1(p,z(p,\lambda))+\frac{L(p,z(p,\lambda))}{2}\frac{\partial\Theta_1}{\partial z}(p,z(p,\lambda))\frac{1}{\frac{\partial L}{\partial z}(p,z(p,\lambda))}\\&\sim \frac{1}{(p-2)^{3/2}} \to \infty,
    \end{align*}
    and the conclusion follows.
\end{proof}
\begin{proof}[Proof of Theorem \ref{thm:asymptotics}-(a).] Let $\lambda^*=\lambda^*(p)=(L(p,1))^2$. Then, by Lemma \ref{lem:lim_z_p2}, $\lambda^*=(L(p,1))^2\sim \infty$ as $p\sim 2^+$. In particular, for $p$ close to $2$, $\lambda \neq \lambda^*(p)$. The result now follows from \eqref{eq:sinalp2}, Section \ref{sec:GSS_particular} and Theorem \ref{th:5.27}.
\end{proof}

\subsubsection{Asymptotics for $p\to \infty$}
{Observe that, for $\alpha$ given by \eqref{eq:alpha},
\[
\lim_{p\to \infty} \alpha=\lim_{p\to \infty} \frac{p-6}{2(p-2)}=\frac{1}{2}.
\]
In particular, contrary to the case $p\sim 2^+$, where the first term in \eqref{der_Theta_z2} dominated, both terms in \eqref{der_Theta_z2} will be of equal size as $p\sim \infty$. This fact will force us to perform a more precise asymptotic analysis.}
\begin{lemma}\label{lemma1:pinfty}
    Take $\delta>0$. Then, as $p\sim\infty$,
    \begin{equation}
        f(z)\sim \frac{z^2}{2},\quad f'(z) \sim z,\quad f''(z)\sim 1, \quad f_2^{-1}((1-\theta^2)f(z))\sim 1. 
    \end{equation}
    uniformly for $z\in [\delta,1-\delta]$,
\end{lemma}
\begin{proof}
    The first three asymptotics follow directly from the expression for $f(z)=\frac{z^2}{2}-\frac{z^p}{p}$. For the latter, observe that, by the definition of $f_2$, $f_2^{-1}((1-\theta^2)f(z))>1$. If 
    $$\liminf_{p\to \infty} f_2^{-1}((1-\theta^2)f(z))\ge 1+\epsilon,\quad \mbox{ for some } \epsilon>0,$$
    then we would have
    $$
    -\theta^2f(1)<\liminf_{p\to \infty} (1-\theta^2)f(z)\le \liminf_{p\to \infty}f(1+\epsilon) = -\infty,
    $$
    which is a contradiction.
\end{proof}
\begin{lemma}\label{lem:pinfty}
     Take $\delta>0$. Given $z\in[0,1-\delta]$, set $\sigma=-(1-\theta^2)z^2$. Then, as $p\sim\infty$,
     \begin{equation}\label{eq:Lpinfty}
         L(p,z)\sim \int_z^1 \frac{dt}{\sqrt{t^2+\sigma}}=\frac{1}{2}\ln\left(\frac{(\sqrt{1+\sigma}+1)(\theta-1)}{(\sqrt{1+\sigma}-1)(\theta+1)}\right)=:G(\sigma),\quad  \frac{\partial L}{\partial z}(p,z)\sim -\frac{1}{z\sqrt{1+\sigma}},
     \end{equation}
     \begin{equation}\label{eq:mu1pinfty} \mu_1(p,z,\theta)\sim \frac{\sqrt{1+\sigma}}{2} - \frac{\theta z^2}{2} - \frac{\sigma G(\sigma)}{2},\quad \frac{\partial \mu_1}{\partial z}(p,z,\theta)\sim -\theta z - \frac{G(\sigma)\sigma}{z}+ \frac{\sigma}{z\sqrt{\sigma+1}},
     \end{equation}
     \begin{equation}\label{eq:mu2pinfty}
          \mu_2(p,z)\sim \frac{z^2}{2},\quad \frac{\partial \mu_2}{\partial z}(p,z) \sim z,
     \end{equation}
     uniformly in $z\in[\delta,1-\delta]$. In particular, if $K=1/2$ for the $\mathcal{T}$ graph and $K=1$ for the tadpole graph,
     \begin{equation}\label{eq:thetapinfty}
         \Theta_1(p,z)\sim K\left(\sqrt{1+\sigma}-\sigma G(\sigma)\right), \quad \frac{\partial \Theta_1}{\partial z}(p,z) \sim 2K\left( - \frac{G(\sigma)\sigma}{z}+ \frac{\sigma}{z\sqrt{\sigma+1}} \right)
     \end{equation}
     and, for $\alpha=\frac{6-p}{p-2}$,
     \begin{equation}\label{eq:massapinfty}
          \alpha\Theta_1(p,z)\frac{\partial L}{\partial z}(p,z)+\frac{L(p,z)}{2}\frac{\partial\Theta_1}{\partial z}(p,z)\sim  \frac{K\sigma}{z}G(\sigma)^2\frac{d}{d\sigma}\left(\frac{\sqrt{1+\sigma}}{G(\sigma)}-\sigma\right).
     \end{equation}
\end{lemma}
\begin{proof}
    The proof follows from arguments similar to those used to prove Lemma \ref{assymptotic_mu_der}. 
    
    We start with the asymptotics for $L$. Since $z$ is bounded away from $0$ and $\varphi(0)$ is a simple zero of $f$, by dominated convergence, the asymptotics derived in Lemma \ref{lemma1:pinfty} yield
    $$
    L(p,z)=\frac{1}{\sqrt{2}}\int_z^{f_2^{-1}((1-\theta^2)f(z))}\frac{dt}{\sqrt{f(t)-(1-\theta^2)f(z)}} \sim \int_z^{1}\frac{dt}{\sqrt{t^2-(1-\theta^2)z^2}} = \int_z^{1}\frac{dt}{\sqrt{t^2+\sigma}}.  
    $$
    Computing explicitly this integral, we arrive at the first equality in \eqref{eq:Lpinfty}. For the second, recall from Lemma \ref{lemma:Theta_1_expressions} that
    \begin{equation}\label{eq:aux_asym11}
\sqrt{2}\frac{\partial L}{\partial z}(p,z)\left((1-\theta^2)f(z)-f(1)\right)=\theta\frac{f(1)}{\sqrt{f(z)}}+(1-\theta^2)f'(z)\int_z^{f_{2}^{-1}\left((1-\theta^2)f(z)\right)}\frac{A(t)dt}{\sqrt{f(t)-(1-\theta^2)f(z)}},
\end{equation}
where $A(t)=\frac{1}{2}-\frac{f(t)-f(1)}{(f'(t))^2}f''(t)$. Observing that $A(t)\sim \frac{1}{2t^2}$,
    \begin{align*}
        -\frac{1}{\sqrt{2}}\frac{\partial L}{\partial z}(p,z)(\sigma +1) = \frac{\theta}{\sqrt{2}z}-\frac{\sqrt{2}\sigma}{z}\int_z^1 \frac{dt}{2t^2\sqrt{t^2+\sigma}} = \frac{\theta}{\sqrt{2}z} + \frac{\sigma}{\sqrt{2}z}\left[\frac{\sqrt{t^2+\sigma}}{t\sigma}\right]_z^1 = \frac{\sqrt{1+\sigma}}{\sqrt{2}z}
    \end{align*}
    which concludes the proof of \eqref{eq:Lpinfty}. Next, we have
\begin{align*}
        \mu_1(p,z,\theta)&=\frac{1}{\sqrt{2}}\int_z^{f_2^{-1}((1-\theta^2f(z)))}\frac{t^2 dt}{\sqrt{f(t)-(1-\theta^2)f(z)}}\\&\sim \int_z^{1}\frac{t^2dt}{\sqrt{t^2-(1-\theta^2)z^2}} \\&= \left[\frac{t\sqrt{t^2+\sigma}}{2}\right]_{z}^1 - \frac{\sigma}{2}\int_z^1 \frac{dt}{\sqrt{t^2-(1-\theta^2)z^2}} = \frac{\sqrt{1+\sigma}}{2}-\frac{\theta z^2}{2} - \frac{\sigma G(\sigma)}{2}.
\end{align*}
From \eqref{mu1DerAux},
\begin{align*}
    		\sqrt{2}\frac{\partial\mu_1}{\partial z}(p,z,\theta)&((1-\theta^2)f(z)-f(1))=\theta z^2\frac{f(1)}{\sqrt{f(z)}}\nonumber\\
        &-(1-\theta^2)f'(z)\int_z^{f_{2}^{-1}((1-\theta^2)f(z))}\left(t^2-\frac{g(t)}{2}\right)\frac{dt}{\sqrt{f(t)-(1-\theta^2)f(z)}},\label{mu1DerAux}
\end{align*}
and thus
\begin{align*}
    -\frac{1+\sigma}{\sqrt{2}}\frac{\partial \mu_1}{\partial z}(p,z,\theta)&\sim \frac{\theta z}{\sqrt{2}} - \frac{(1-\theta^2)z}{\sqrt{2}}\int_z^1 \left(\frac{1}{2}-t^2\right)\frac{dt}{\sqrt{t^2+\sigma}}\\&= \frac{\theta z}{\sqrt{2}} + \frac{\sigma}{\sqrt{2}z} \left((\sigma+1)G(\sigma)-\sqrt{\sigma+1}+\theta z^2\right)\\&=-\frac{1+\sigma}{\sqrt{2}}\left(-\theta z - \frac{G(\sigma)\sigma}{z}+ \frac{\sigma}{z\sqrt{\sigma+1}}\right),
\end{align*}
which proves \eqref{eq:mu1pinfty}. Finally, \eqref{eq:mu2pinfty} follows from 
$$
\mu_2(p,z)=\frac{1}{\sqrt{2}}\int_0^z \frac{t^2dt}{\sqrt{f(t)}} \sim \int_0^t t dt = \frac{z^2}{2}\quad, \frac{\partial \mu_2}{\partial z}=\frac{z^2}{\sqrt{2f(z)}}\sim z.
$$
The proof of \eqref{eq:thetapinfty}  follows from using \eqref{Theta_1_T'} (for the $\mathcal{T}$ graph) and \eqref{Theta_1_TADPOLE'} (for the tadpole graph). Finally, since $\alpha \to -1/2$, \eqref{eq:massapinfty} is a consequence of
\begin{equation}
    \alpha\Theta_1(p,z)\frac{\partial L}{\partial z}(p,z)+\frac{L(p,z)}{2}\frac{\partial\Theta_1}{\partial z}(p,z) \sim \frac{L^2(p,z)}{2} \frac{\partial}{\partial z}\left(\frac{\Theta_1(p,z)}{L(p,z)}\right)\sim  \frac{K\sigma}{z}G(\sigma)^2\frac{d}{d\sigma}\left(\frac{\sqrt{1+\sigma}}{G(\sigma)}-\sigma\right).
\end{equation}
\end{proof}
\begin{proposition}
  For $z\in[\delta,1-\delta]$, as $p\sim \infty$,
    \begin{equation}\label{eq:sinalpinfty}
        \alpha\Theta_1(p,z)\frac{\partial L}{\partial z}(p,z)+\frac{L(p,z)}{2}\frac{\partial\Theta_1}{\partial z}(p,z)>0.
    \end{equation}
    In particular, as $p\sim \infty$,
\begin{equation}\label{eq:sinalpinfty2}
        \frac{\partial \Theta}{\partial \lambda}(p,\lambda) < 0, \quad  \mbox{ uniformly in }\lambda\in[\delta,1/\delta].
\end{equation}
\end{proposition}
\begin{proof}
\textit{Step 1. Proof of \eqref{eq:sinalpinfty}.} By Lemma \ref{lem:pinfty}, the claim is equivalent to
$$
\sigma \frac{d}{d\sigma}\left(\frac{\sqrt{1+\sigma}}{G(\sigma)}-\sigma\right)>0.
$$
Fix first the case of the $\mathcal{T}$ graph, where $\theta=2$. Then $\sigma=-(1-\theta^2)z^2>0$. For $\xi=\sqrt{1+\sigma}=\sqrt{1+3z^2}\in (1,2)$, we want to check that
\begin{align*}
    0<\frac{d}{d\xi}\left(\frac{\xi}{\ln\left(\frac{\xi+1}{3(\xi-1)}\right)}-\frac{\xi^2-1}{2}\right)&= \frac{1}{\ln^2\left(\frac{\xi+1}{3(\xi-1)}\right)}\left(\ln\left(\frac{\xi+1}{3(\xi-1)}\right)+\frac{2\xi}{\xi^2-1}-\xi\ln^2\left(\frac{\xi+1}{3(\xi-1)}\right)\right)\\&=:\frac{1}{\ln^2\left(\frac{\xi+1}{3(\xi-1)}\right)}h(\xi).
\end{align*}
A direct computation shows that
$$
h'(\xi)=-\left(\ln\left(\frac{\xi+1}{3(\xi-1)}\right)-\frac{2\xi}{\xi^2-1}\right)^2<0,\quad h(2)=\frac{4}{3}.
$$
Therefore $h>0$ over $(1,2)$, which concludes the proof of \eqref{eq:sinalpinfty} for the $\mathcal{T}$ graph. For the tadpole graph, $\theta=1/2$ and $\sigma=-(1-\theta^2)z^2<0.$ Writing once again $\xi=\sqrt{1+\sigma}=\sqrt{1-\frac{3}{4}z^2}\in (0,1)$, we need to verify that
\begin{align*}
    0>\frac{d}{d\xi}\left(\frac{\xi}{\ln\left(\frac{\xi+1}{3(1-\xi)}\right)}-\frac{\xi^2-1}{2}\right)&= \frac{1}{\ln^2\left(\frac{\xi+1}{3(1-\xi)}\right)}\left(\ln\left(\frac{\xi+1}{3(1-\xi)}\right)+\frac{2\xi}{\xi^2-1}-\xi\ln^2\left(\frac{\xi+1}{3(1-\xi)}\right)\right)\\&=:\frac{1}{\ln^2\left(\frac{\xi+1}{3(1-\xi)}\right)}j(\xi).
\end{align*}
Again, by direct computation,
$$
j'(\xi)=-\left(\ln\left(\frac{\xi+1}{3(1-\xi)}\right)-\frac{2\xi}{\xi^2-1}\right)^2<0,\quad j(0)=\ln\frac{1}{3}<0
$$
and thus $j<0$ over $(0,1)$. This concludes the proof of \eqref{eq:sinalpinfty}.

\textit{Step 2.} We claim that, letting $z(p,\lambda)$ be the (unique) solution to $\sqrt{\lambda}=L(p,z)$, there exists $\delta>0$ such that $z(p,\lambda)\in [\delta,1-\delta]$, uniformly in $\lambda\ge \lambda_0$. 

By contradiction, suppose that there exists a sequence $(p_n,\lambda_n)$, with $p_n\to \infty$ and $\lambda_0<\lambda_n$, such that  $z(p_n,\lambda_n)\ge 1$. As $z\mapsto G(-(1-\theta^2)z)$ is strictly decreasing, with $G(0^+)=+\infty$ and $G(1)=0$, we may let $z_0\in (0,1)$ be the unique solution to $G(-(1-\theta^2)z)=\sqrt{\lambda_0}$. Then, for $n$ large, $z(p_n,\lambda_n)>\frac{z_0+1}{2}>z_0$. Since, for each fixed $p\in (2,\infty)$, $z\mapsto L(p,z)$ is strictly decreasing,
$$
\sqrt{\lambda_n}=L(p_n,z(p_n,\lambda_n)) < L\left(p_n,\frac{z_0+1}{2}\right)
$$
and, passing to the limit,
$$
\sqrt{\lambda_0} \le \lim L\left(p_n,\frac{z_0+1}{2}\right) =  G\left(-(1-\theta^2)\frac{z_0+1}{2}\right) <  G\left(-(1-\theta^2)z_0\right) =\sqrt{\lambda_0},
$$
which is absurd. This proves that $z(p,\lambda)<1-\delta$, for some $\delta>0$. A similar argument can be used to show that $z(p,\lambda)>\delta$ and the claim follows.

\textit{Step 3.} For $\lambda>0$ fixed, we need to prove that, as $p\sim \infty$, the expression in \eqref{der_Theta_z}, with $z=z(p,\lambda)$, is strictly negative. Since $\frac{\partial L}{\partial z}<0$, it is enough to check that 
\begin{equation}
    \alpha\Theta_1(p,z(p,\lambda))\frac{\partial L}{\partial z}(p,z(p,\lambda))+\frac{L(p,z(p,\lambda))}{2}\frac{\partial\Theta_1}{\partial z}(p,z(p,\lambda))>0,
\end{equation}
which follows directly from Steps 1 and 2.
\end{proof}

\begin{proof}{Proof of Theorem \ref{thm:asymptotics}-(b).} Let $\lambda^*=\lambda^*(p)=(L(p,1))^2$. Then, arguing as in the proof of Lemma \ref{lem:pinfty}, $\lambda^*\sim 0$ as $p\sim \infty$. In particular, for $p$ large, $\lambda \neq \lambda^*(p)$. The result now follows from \eqref{eq:sinalpinfty2} and the Grillakis-Shatah-Strauss theory developed in Section \ref{sec:GSS_particular}.
\end{proof}

\paragraph{Acknowledgments.}
F. Agostinho, S. Correia and H. Tavares are partially supported by the Portuguese government through FCT - Funda\c c\~ao para a Ci\^encia e a Tecnologia, I.P., project UIDB/04459/2020 with DOI identifier 10-54499/UIDP/04459/2020 (CAMGSD). F. Agostinho was also partially supported by Funda\c c\~ao para a Ci\^encia e Tecnologia, I.P., through the PhD grant UI/BD/150776/2020. H. Tavares is also partially supported by FCT under the project  2023.13921.PEX, with DOI identifier  https://doi.org/10.54499/2023.13921.PEX (project SpectralOPs).

	{\noindent Francisco Agostinho, Sim\~ao Correia and Hugo Tavares}\\
{\footnotesize
	Center for Mathematical Analysis, Geometry and Dynamical Systems,\\
	Instituto Superior T\'ecnico, Universidade de Lisboa\\
    Department of Mathematics,\\
	Av. Rovisco Pais, 1049-001 Lisboa, Portugal\\
	\texttt{francisco.c.agostinho@tecnico.ulisboa.pt}\\ \texttt{simao.f.correia@tecnico.ulisboa.pt}\\ \texttt{hugo.n.tavares@tecnico.ulisboa.pt}
}

\end{document}